\documentclass[11pt]{amsart}

\usepackage[usenames,dvipsnames,svgnames,table]{xcolor}
\usepackage[colorlinks=true, pdfstartview=FitV, linkcolor=blue, citecolor=blue, urlcolor=darkblue]{hyperref}

\usepackage{geometry}                % See geometry.pdf to learn the layout options. There are lots.
\usepackage{graphicx}
\usepackage{mathrsfs}
\usepackage{amssymb}
\usepackage{amsfonts}
\usepackage{epstopdf}
\usepackage{lscape}
\usepackage[utf8]{inputenc}
\usepackage{tikz,caption}

\usepackage{enumitem}
\setlist[itemize]{leftmargin=2em}
\setlist[enumerate]{leftmargin=2em}
\usepackage{multirow}
\usepackage{mathtools}
\usepackage[linesnumbered,ruled]{algorithm2e}

\definecolor{darkblue}{rgb}{0.0,0,0.7} % darkblue color
\definecolor{darkred}{rgb}{0.7,0,0} % darkred color
\definecolor{darkgreen}{rgb}{0, .6, 0} % darkgreen color

\newcommand{\defncolor}{\color{darkred}}
\newcommand{\defn}[1]{{\defncolor\emph{#1}}} % emphasis of a definition

\newtheorem{theorem}{Theorem}[section]
\newtheorem{prop}[theorem]{Proposition}
\newtheorem{cor}[theorem]{Corollary}
\newtheorem{lemma}[theorem]{Lemma}

\theoremstyle{definition}
\newtheorem{definition}[theorem]{Definition}
\newtheorem{example}[theorem]{Example}
\newtheorem{remark}[theorem]{Remark}
\numberwithin{equation}{section}

\newcommand{\idiot}[1]{\vspace{5 mm}\par \noindent
\marginpar{\textsc{Note}}
\framebox{\begin{minipage}[c]{0.95 \textwidth}
#1 \end{minipage}}\vspace{5 mm}\par}

\renewcommand{\idiot}[1]{}

\def\o{\overline}

\def\vtype{{\overrightarrow{\mathsf{type}}}}

\newcommand{\cf}{\textit{cf.} }

\usepackage{amsmath}
\usepackage{amssymb}
\usepackage{amsfonts}
\usepackage{amsthm}
\usepackage{mathtools}
\usepackage{enumitem}
\setlist[itemize]{label=\small\textbullet}
\setlist{leftmargin=*}
\setlist[1]{labelindent=0.5\parindent}
\setenumerate{listparindent=\parindent}

\usepackage{booktabs}
\usepackage{multirow}
\usepackage{multicol}

\usepackage{mathrsfs}

\usepackage{environ}
\NewEnviron{MarginNote}[1][0px]{\marginnote[#1]{\BODY}}
\NewEnviron{MarginRemark}[1][0px]{\marginnote[#1]{\BODY}}
\NewEnviron{SideNote}[1][]{\sidenote[][#1]{\RaggedRight\BODY}\xspace}

\usepackage{framed}

{\color{#1}\MakeFramed{\advance\hsize-\width \FrameRestore}\small}%
{\endMakeFramed}

\usepackage{rotating}
\usepackage{ytableau}

\usepackage{tikz}
\usepackage{tikz-cd}
\usetikzlibrary{calc}
\usetikzlibrary{arrows}
\usepackage{tikzsymbols}

\usepackage{ytableau}
\ytableausetup{smalltableaux, aligntableaux=center, textmode, boxsize=0.6em}

\def\unprotectedboldentry#1{\textcolor{Red}{\large{#1}}}
\def\boldentry{\protect\unprotectedboldentry}
\usepackage{tikz}
\usetikzlibrary{calc}
\usepackage{ifthen}
\newcommand{\tikztableauinternal}[1]{
    \def\newtableau{#1}
    \coordinate (x) at (-0.5,0.5);
    \coordinate (y) at (-0.5,0.5);
    \foreach \row in \newtableau {
        \foreach \entry in \row {
            \ifthenelse{\equal{\entry}{X} \OR \equal{\entry}{None}}
               {
                \node (y) at ($(y) + (1,0)$) {};
                \fill[color=gray!30] ($(y)-(0.5,0.5)$) rectangle +(1,1);
                \draw[color=gray, dotted] ($(y)-(0.5,0.5)$) rectangle +(1,1);
               }
               {
                \ifthenelse{\equal{\entry}{\boldentry X}}
                   {
                    \node (y) at ($(y) + (1,0)$) {};
                    \fill[color=gray] ($(y)-(0.5,0.5)$) rectangle +(1,1);
                    \draw ($(y)-(0.5,0.5)$) rectangle +(1,1);
                   }
                   {
                    \node (y) at ($(y) + (1,0)$) {\entry};
                    \draw ($(y)-(0.5,0.5)$) rectangle +(1,1);
                   }
               }
            }
        \coordinate (x) at ($(x)+(0,1)$);
        \coordinate (y) at (x);
        }
}

\def\sectionautorefname~#1\null{\S#1\null}
\def\subsectionautorefname~#1\null{\S#1\null}
\def\subsubsectionautorefname~#1\null{\S#1\null}
\def\equationautorefname~#1\null{Equation~(#1)\null}
\def\itemautorefname~#1\null{#1\null}
\def\figureautorefname~#1\null{Figure~#1\null}
\definecolor{darkblue}{rgb}{0.0,0,0.7} % darkblue color
\definecolor{darkred}{rgb}{0.7,0,0} % darkred color
\definecolor{darkgreen}{rgb}{0, .6, 0} % darkgreen color

\usepackage{xargs}
\usepackage[colorinlistoftodos]{todonotes}

\newcommand{\UU}{\mathcal{U}}
\newcommand{\J}{\mathscr{J}}

\newcommand{\precdot}{\mathrel{\prec\hskip-0.5em\cdot}}
\newcommand{\smallest}{\varsigma}
\newcommand{\sym}{\mathfrak{S}}

\title[Submonoids of uniform block permutations]{The lattice of submonoids of the uniform block permutations containing the symmetric group}

\author[Orellana]{Rosa Orellana}
\address[R. Orellana]{Mathematics Department, Dartmouth College,
Hanover, NH 03755, U.S.A.}
\email{Rosa.C.Orellana@dartmouth.edu}
\urladdr{\href{https://math.dartmouth.edu/~orellana/}{https://math.dartmouth.edu/~orellana/}}

\author[Saliola]{Franco Saliola}
\address[F. Saliola]{LACIM, D\'epartement de math\'ematiques,
Universit\'e du Qu\'ebec \`a Montr\'eal, Canada}
\email{saliola.franco@uqam.ca}
\author[Schilling]{Anne Schilling}
\address[A. Schilling]{Department of Mathematics, University of California, One Shields
Avenue, Davis, CA 95616-8633, U.S.A.}
\email{anne@math.ucdavis.edu}
\urladdr{\href{http://www.math.ucdavis.edu/~anne}{http://www.math.ucdavis.edu/~anne}}

\author[Zabrocki]{Mike Zabrocki}
\address[M. Zabrocki]{Department of Mathematics and Statistics,  York University, 4700 Keele Street, Toronto,
Ontario M3J 1P3, Canada}
\email{zabrocki@mathstat.yorku.ca}
\urladdr{\href{http://garsia.math.yorku.ca/~zabrocki/}{http://garsia.math.yorku.ca/~zabrocki/}}


\begin{document}

\date{\today}

\begin{abstract}
We study the lattice of submonoids of the uniform block permutation monoid containing the symmetric group (which is its group of units).
We prove that this lattice is distributive under union and intersection by relating the submonoids containing the symmetric group
to downsets in a new partial order on integer partitions. Furthermore, we show that the sizes of the $\mathscr{J}$-classes of the uniform block
permutation monoid are sums of squares of dimensions of irreducible modules of the monoid algebra.
\end{abstract}

\maketitle

% \sidebarseparator

% \vspace{.5in}
% \setcounter{tocdepth}{3}
% \tableofcontents

%%%%%%%%%%%%%%%%%%%%%%%%%%%%%%%%%%%%%%%%%%%%%%%%%%%%%%%%%%%%%%
\section{Introduction}
%%%%%%%%%%%%%%%%%%%%%%%%%%%%%%%%%%%%%%%%%%%%%%%%%%%%%%%%%%%%%%

A block permutation is a bijection between the set of blocks of a set partition $A$ of $[k]=\{1,2,\ldots,k\}$ and the set of 
blocks of a set partition $B$ of $[\o{k}]=\{\o{1},\o{2},\ldots,\o{k}\}$.
A block permutation is uniform if the image of each block $A_i$ of  $A$ has the same cardinality as $A_i$.
There are at least three approaches to the algebra
of uniform block permutations that are established
in the literature.  The first is as a factorizable
inverse semigroup \cite{FitzGerald.2003, FitzGeraldLeech.1998, FitzGerald.2010}.
The second is as a centralizer algebra \cite{Kosuda.2000, Kosuda.2001, Kosuda.2006, Naruse.2005},
where it is a limiting case of the Tanabe algebra \cite{Tanabe.1997}.
The third is as a Hopf algebra \cite{AguiarOrellana.2008}
where the set of all uniform block permutations are graded by size
and there is an external product and coproduct.

In \cite{OSSZ.2022}, we 
studied the representation theory of the algebra of uniform block permutations as a monoid algebra
in the pursuit of computing the restriction of an irreducible $GL_n$ module to the symmetric group $\sym_n$ of
permutation matrices, where note that $\sym_n \subseteq GL_n$.

For a positive integer $k$, let $\UU_k$ denote the
monoid of uniform block permutations of the set $[k] \cup [\o{k}]$
with the usual diagram algebra product.  The irreducible
representations of $\UU_k$ are indexed by sequences of partitions
$\vec{\lambda} = (\lambda^{(1)}, \lambda^{(2)}, \ldots, \lambda^{(k)})$
such that $|\lambda^{(1)}| + 2 | \lambda^{(2)}| + \cdots + k |\lambda^{(k)}| = k$.
Let $W^{\vec{\lambda}}_{\UU_k}$ denote
the irreducible $\UU_k$-module indexed by $\vec{\lambda}$.
One of the main results of \cite{OSSZ.2022} is a formula
in terms of the operation of plethysm on symmetric functions
for the multiplicity of an irreducible symmetric group module
in ${\mathrm {Res}}^{\UU_k}_{\sym_k} W^{\vec{\lambda}}_{\UU_k}$~.

Since both $\UU_k$ and $\sym_k$ are monoids,
it is natural to ask for a description of the
monoids which lie in between $\sym_k$ and $\UU_k$.
The following theorem summarizes the results in this paper and provides
a precise description.

\begin{theorem}[Theorem~\ref{union-is-submonoid}] 
 \label{theorem:summary}
For $k$ a positive integer, the set of submonoids of $\UU_k$ that contain $\sym_k$
forms a distributive lattice with the operations of union and intersection.
\vskip .1in

\noindent
{\rm (Corollary~\ref{cor:submonoid characterization})}
Define a partial order on partitions $\mu, \lambda$ of $k$ that are not equal
to $1^k$ such that $\mu \preceq \lambda$ if and only if $\mu$ is coarser than $\lambda$
in refinement order and $\smallest_\mu > \smallest_\lambda$, where $\smallest_\mu$ denotes
the smallest part of $\mu$ not equal to $1$.
For $\mu$ a partition of $k$, let $J_\mu$ denote a $\mathscr{J}$-class of $\UU_k$.
Every submonoid $S$ of $\UU_k$ that contains $\sym_k$ is of the form
\[
    S = \sym_k \cup \bigcup_{\mu \in I} J_\mu
\]
for some down set $I$ of the partial order $\preceq$.
\end{theorem}

The group of units of a monoid $M$ is the
largest group, $G \subseteq M$, containing the identity
element of the monoid. The group of units of
the uniform block permutation monoid $\UU_k$ is the symmetric
group $\sym_k$.
The poset of submonoids of $M$ which contain $G$ \cite{ShevrinOvsyannikov.2013} does 
not seem to have such nice properties as in the case of $\UU_k$.
For instance, our experiments calculating the
lattice of monoids of the dual-inverse semigroup \cite{FitzGeraldLeech.1998}
indicate that this lattice of
submonoids for $k = 5$ does not form a distributive lattice.

The paper is organized as follows.
In Section~\ref{section:notation}, we define the necessary notation.
Sections~\ref{section:new partial order} and \ref{section:cover relations} characterize the new 
partial order on integer partitions used in Theorem~\ref{theorem:summary}.
The connection between the partial order
and the lattice of submonoids of $\UU_k$ containing
symmetric group $\sym_k$ is completed in Section \ref{section:submonoids}.
Section \ref{section:remarks} concludes with
some remarks about the consequences of this result.
In particular, we show that the sizes of the $\mathscr{J}$-classes of the uniform block
permutation monoid are sums of squares of dimensions of irreducibles.

%%%%%%%%%%%%%%%%%%%%%%%%%%%%%%%%%%%%%%%%%%%%%%%%%%%%%%%%%%%%%%%%%%%%%%%%
\section*{Acknowledgement}
We thank James East, Jinting Liang, Stuart Margolis, John Rhodes and Bruce Sagan for enlightening discussions.

This work was partially supported by NSF grants DMS--2153998, DMS--2053350 and NSERC/CNSRG.

%%%%%%%%%%%%%%%%%%%%%%%%%%%%%%%%%%%%%%%%%%%%%%%%%%%
\section{Notation}
\label{section:notation}
%%%%%%%%%%%%%%%%%%%%%%%%%%%%%%%%%%%%%%%%%%%%%%%%%%%

In this section, we introduce notation related to the uniform block permutation monoid.  We
follow the presentation in~\cite{OSSZ.2022} and refer the reader to this reference for additional
details and references. Fix a positive integer $k$ in this section.

We say that $\lambda = (\lambda_1, \lambda_2, \ldots, \lambda_r)$
is a \defn{partition} of $k$ if for all $1 \leqslant i <r$,
$\lambda_i$ is a positive integer, $\lambda_i\geqslant \lambda_{i+1}$, $\lambda_r>0$ and $\sum_{i=1}^r \lambda_i = k$. 
If $\lambda$ is a partition of $k$, we write $\lambda\vdash k$.
We denote by $\mathcal{P}_k = \{ \lambda \vdash k\}$ the set of partitions of $k$.
Let $\ell(\lambda):=r$ denote the length of the partition.
We also use exponential notation to represent partitions. If the partition has $b$ occurrences of an integer $i$,
we represent it by $i^b$ (e.g. $\lambda = (4,4,4,2,1,1,1,1)$ is also represented by $1^42^14^3$).

For a finite set $X$, we say that $\pi$ is a \defn{set partition} of $X$ (indicated by
$\pi \vdash X$) if $\pi = \{ A_1, A_2, \ldots, A_r \}$,
where $\emptyset \subsetneq A_i \subseteq X$ for $1 \leqslant i \leqslant r$,
$A_i \cap A_j = \emptyset$ if $i \neq j$, and $A_1 \cup A_2 \cup \cdots \cup A_r = X$.
Denote by $\ell(\pi)=r$ the length of the set partition.
The \defn{type} of a set partition $\pi$ is
$\mathsf{type}(\pi) = \mathsf{sort}( |A| : A \in \pi)$, where the operation of
$\mathsf{sort}$ arranges the integers in the list in weakly decreasing order so that
the result is a partition.

We say that $\pi$ is \defn{finer} than $\tau$ (equivalently, $\tau$ is \defn{coarser} than $\pi$)
if for every block $A \in \pi$, there is a block $B \in \tau$ such that $A \subseteq B$.
The set of set partitions of $X$ forms a lattice under refinement order, where the join operation $\pi \vee \tau$ is defined to be 
the finest set partition coarser than both $\pi$ and $\tau$ and the meet operation $\pi \wedge \tau$ 
is the coarsest set partition which is finer than both $\pi$ and $\tau$.

Let $[k] := \{ 1, 2, \ldots, k \}$ and $[\o{k}] := \{ \o{1}, \o{2}, \ldots, \o{k} \}$.
We say that a set partition $\pi$ of $[k] \cup [\o{k}]$ is
\defn{uniform} if for each $A \in \pi$ the condition $|A \cap [k] | = |A \cap [\o{k}]|$ holds.
Let $\UU_k$ denote the set of uniform set partitions of $[k] \cup [\bar{k}]$.
For reference, we mention that the number of elements in $\UU_k$ is sequence
\href{https://oeis.org/A023998}{A023998}
in the Online Encyclopedia of Integer Sequences \cite{OEIS} and
(starting with $k=0$) the first few terms of the sequence are
\[
1, 1, 3, 16, 131, 1496, 22482, 426833,\ldots~.
\]

For a set $S \subseteq [k] \cup [\o{k}]$, let
$\o{S} = \{ i : \o{i} \in S \cap [\o{k}]\} \cup \{ \o{i} : i \in S \cap [k]\}$.
For $\pi \in \UU_k$, let
$\mathsf{top}(\pi)$ be the set partition of $[k]$ consisting of the blocks $A \cap [k]$
for $A \in \pi$ and $\mathsf{bot}(\pi)$ the set
partition of $[k]$ containing the blocks $\overline{A} \cap [k]$
for $A \in \pi$.
For $\pi \in \UU_k$, we use $\mathsf{type}(\pi)$ to denote
$\mathsf{type}(\mathsf{top}(\pi))$, which is a partition of the integer $k$.

\begin{example}
We have
\begin{align*}
    \pi & = \{ \{1,3,\bar{1},\bar{2} \}, \{2,\bar{4} \},
    \{4,6,\bar{3}, \bar{6}\}, \{5,\bar{7} \}, \{7, 8, 9, \bar{5}, \bar{8}, \bar{9} \}  \} \in \UU_9, \\
\mathsf{top}(\pi) &=\{\{ 1,3\}, \{2\}, \{4, 6\},  \{5\}, \{7, 8, 9\}\}, \\
\mathsf{bot}(\pi) &=\{ \{1,2\}, \{ 4\}, \{3,6\}, \{7\}, \{5, 8, 9\}\}, \\
\mathsf{type}(\pi) &= (3,2,2,1,1).
\end{align*}
\end{example}

We will also view elements of $\UU_k$ as a graph and use this graph structure to define
a monoid product on these elements.  We draw the graphs in two rows, the elements $[k] = \{ 1,2, \ldots, k \}$
are arranged from left to right in a top row and the elements $[\o{k}] = \{ \o1,\o2, \ldots, \o{k} \}$
from left to right in a bottom row.  An element $\pi \in \UU_k$ is represented as a graph,
where two vertices are connected with a path if and only if they are in the same set
of $\pi$. The \defn{diagram} for $\pi$ represents a class of labeled graphs that have the same connected components.

Consider $\pi, \gamma \in \UU_k$ as diagrams.  The monoid
product $\pi \gamma \in \UU_k$ is computed as follows:
Stack the graph $\pi$ on top of $\gamma$ and identify the
bottom vertices of $\pi$ with the top vertices of $\gamma$. Compute the
connected components of the three-row diagram and then eliminate the vertices of the middle row
from this diagram.
The element $\pi \gamma \in \UU_k$ has the same connected components
as this stacked graph.

\begin{example}
Let $\pi = \{ \{1,2,7,\o{2},\o{8},\o{9}\},\{3,\o{1}\},\{4,8,\o{3},\o{5}\},\{5,9,\o{4},\o{6}\},\{6,\o{7}\}\}$
and $\tau =$\\ $\{\{1,\o{2}\}, \{2,3,\o{1},\o{4}\}, \{4,6,\o{3},\o{5}\}, \{5,\o{9}\}, \{7,8,9,\o{6},\o{7},\o{8}\}\}$.
The diagram for $\pi$ is shown below on the left and the diagram for $\tau$ is on the right.
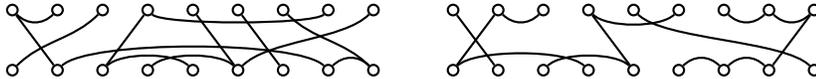
\begin{center}
\begin{tikzpicture}[scale = 0.4,thick, baseline={(0,-1ex/2)}]
\tikzstyle{vertex} = [shape = circle, minimum size = 4pt, inner sep = 1pt]
\node[vertex] (G--9) at (12.0, -1) [shape = circle, draw] {};
\node[vertex] (G--8) at (10.5, -1) [shape = circle, draw] {};
\node[vertex] (G--2) at (1.5, -1) [shape = circle, draw] {};
\node[vertex] (G-1) at (0.0, 1) [shape = circle, draw] {};
\node[vertex] (G-2) at (1.5, 1) [shape = circle, draw] {};
\node[vertex] (G-7) at (9.0, 1) [shape = circle, draw] {};
\node[vertex] (G--7) at (9.0, -1) [shape = circle, draw] {};
\node[vertex] (G-6) at (7.5, 1) [shape = circle, draw] {};
\node[vertex] (G--6) at (7.5, -1) [shape = circle, draw] {};
\node[vertex] (G--4) at (4.5, -1) [shape = circle, draw] {};
\node[vertex] (G-5) at (6.0, 1) [shape = circle, draw] {};
\node[vertex] (G-9) at (12.0, 1) [shape = circle, draw] {};
\node[vertex] (G--5) at (6.0, -1) [shape = circle, draw] {};
\node[vertex] (G--3) at (3.0, -1) [shape = circle, draw] {};
\node[vertex] (G-4) at (4.5, 1) [shape = circle, draw] {};
\node[vertex] (G-8) at (10.5, 1) [shape = circle, draw] {};
\node[vertex] (G--1) at (0.0, -1) [shape = circle, draw] {};
\node[vertex] (G-3) at (3.0, 1) [shape = circle, draw] {};
\draw[] (G-1) .. controls +(0.5, -0.5) and +(-0.5, -0.5) .. (G-2);
%\draw[] (G-2) .. controls +(1, -1) and +(-1, -1) .. (G-7);
\draw[] (G-7) .. controls +(1, -1) and +(-1, 1) .. (G--9);
\draw[] (G--9) .. controls +(-0.5, 0.5) and +(0.5, 0.5) .. (G--8);
\draw[] (G--8) .. controls +(-1, 1) and +(1, 1) .. (G--2);
\draw[] (G--2) .. controls +(-0.75, 1) and +(0.75, -1) .. (G-1);
\draw[] (G-6) .. controls +(0.75, -1) and +(-0.75, 1) .. (G--7);
\draw[] (G--6) .. controls +(-0.75, 1) and +(0.75, -1) .. (G-5);
\draw[] (G-9) .. controls +(-1, -1) and +(1, 1) .. (G--6);
\draw[] (G--6) .. controls +(-0.6, 0.6) and +(0.6, 0.6) .. (G--4);
%\draw[] (G--4) .. controls +(0.75, 1) and +(-0.75, -1) .. (G-5);
\draw[] (G-4) .. controls +(0.5, -0.5) and +(-0.5, -0.5) .. (G-8);
%\draw[] (G-8) .. controls +(-1, -1) and +(1, 1) .. (G--5);
\draw[] (G--5) .. controls +(-0.6, 0.6) and +(0.6, 0.6) .. (G--3);
\draw[] (G--3) .. controls +(0.75, 1) and +(-0.75, -1) .. (G-4);
\draw[] (G-3) .. controls +(-1, -1) and +(1, 1) .. (G--1);
\end{tikzpicture}
\hskip .3in
\begin{tikzpicture}[scale = 0.4,thick, baseline={(0,-1ex/2)}]
\tikzstyle{vertex} = [shape = circle, minimum size = 4pt, inner sep = 1pt]
\node[vertex] (G--9p) at (12.0, -1) [shape = circle, draw] {};
\node[vertex] (G-5p) at (6.0, 1) [shape = circle, draw] {};
\node[vertex] (G--8p) at (10.5, -1) [shape = circle, draw] {};
\node[vertex] (G--7p) at (9.0, -1) [shape = circle, draw] {};
\node[vertex] (G--6p) at (7.5, -1) [shape = circle, draw] {};
\node[vertex] (G-7p) at (9.0, 1) [shape = circle, draw] {};
\node[vertex] (G-8p) at (10.5, 1) [shape = circle, draw] {};
\node[vertex] (G-9p) at (12.0, 1) [shape = circle, draw] {};
\node[vertex] (G--5p) at (6.0, -1) [shape = circle, draw] {};
\node[vertex] (G--3p) at (3.0, -1) [shape = circle, draw] {};
\node[vertex] (G-4p) at (4.5, 1) [shape = circle, draw] {};
\node[vertex] (G-6p) at (7.5, 1) [shape = circle, draw] {};
\node[vertex] (G--4p) at (4.5, -1) [shape = circle, draw] {};
\node[vertex] (G--1p) at (0.0, -1) [shape = circle, draw] {};
\node[vertex] (G-2p) at (1.5, 1) [shape = circle, draw] {};
\node[vertex] (G-3p) at (3.0, 1) [shape = circle, draw] {};
\node[vertex] (G--2p) at (1.5, -1) [shape = circle, draw] {};
\node[vertex] (G-1p) at (0.0, 1) [shape = circle, draw] {};
\draw[] (G-5p) .. controls +(1, -1) and +(-1, 1) .. (G--9p);
\draw[] (G-7p) .. controls +(0.5, -0.5) and +(-0.5, -0.5) .. (G-8p);
\draw[] (G-8p) .. controls +(0.5, -0.5) and +(-0.5, -0.5) .. (G-9p);
\draw[] (G-9p) .. controls +(-0.75, -1) and +(0.75, 1) .. (G--8p);
\draw[] (G--8p) .. controls +(-0.5, 0.5) and +(0.5, 0.5) .. (G--7p);
\draw[] (G--7p) .. controls +(-0.5, 0.5) and +(0.5, 0.5) .. (G--6p);
%\draw[] (G--6p) .. controls +(0.75, 1) and +(-0.75, -1) .. (G-7p);
\draw[] (G-4p) .. controls +(0.6, -0.6) and +(-0.6, -0.6) .. (G-6p);
%\draw[] (G-6p) .. controls +(-0.75, -1) and +(0.75, 1) .. (G--5p);
\draw[] (G--5p) .. controls +(-0.6, 0.6) and +(0.6, 0.6) .. (G--3p);
\draw[] (G-4p) .. controls +(0.75, -1) and +(-0.75, 1) .. (G--5p);
\draw[] (G-2p) .. controls +(0.5, -0.5) and +(-0.5, -0.5) .. (G-3p);
%\draw[] (G-3p) .. controls +(0.75, -1) and +(-0.75, 1) .. (G--4p);
\draw[] (G--4p) .. controls +(-0.7, 0.7) and +(0.7, 0.7) .. (G--1p);
\draw[] (G--1p) .. controls +(0.75, 1) and +(-0.75, -1) .. (G-2p);
\draw[] (G-1p) .. controls +(0.75, -1) and +(-0.75, 1) .. (G--2p);
\end{tikzpicture}
\end{center}

\noindent
Their product
$\pi \tau = \{\{1,2,4,6,7,8,\o{1},\o{4},\o{6},\o{7},\o{8},\o{9}\},\{3, \o{2}\},\{5,9,\o{3},\o{5}\}\}$
is computed with the graphs:

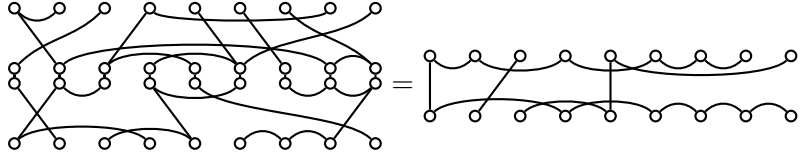
\begin{center}
\raisebox{.25in}{
\begin{tikzpicture}[scale = 0.4,thick, baseline={(0,-1ex/2)}]
\tikzstyle{vertex} = [shape = circle, minimum size = 4pt, inner sep = 1pt]
\node[vertex] (G--9) at (12.0, -1) [shape = circle, draw] {};
\node[vertex] (G--8) at (10.5, -1) [shape = circle, draw] {};
\node[vertex] (G--2) at (1.5, -1) [shape = circle, draw] {};
\node[vertex] (G-1) at (0.0, 1) [shape = circle, draw] {};
\node[vertex] (G-2) at (1.5, 1) [shape = circle, draw] {};
\node[vertex] (G-7) at (9.0, 1) [shape = circle, draw] {};
\node[vertex] (G--7) at (9.0, -1) [shape = circle, draw] {};
\node[vertex] (G-6) at (7.5, 1) [shape = circle, draw] {};
\node[vertex] (G--6) at (7.5, -1) [shape = circle, draw] {};
\node[vertex] (G--4) at (4.5, -1) [shape = circle, draw] {};
\node[vertex] (G-5) at (6.0, 1) [shape = circle, draw] {};
\node[vertex] (G-9) at (12.0, 1) [shape = circle, draw] {};
\node[vertex] (G--5) at (6.0, -1) [shape = circle, draw] {};
\node[vertex] (G--3) at (3.0, -1) [shape = circle, draw] {};
\node[vertex] (G-4) at (4.5, 1) [shape = circle, draw] {};
\node[vertex] (G-8) at (10.5, 1) [shape = circle, draw] {};
\node[vertex] (G--1) at (0.0, -1) [shape = circle, draw] {};
\node[vertex] (G-3) at (3.0, 1) [shape = circle, draw] {};
\draw[] (G-1) .. controls +(0.5, -0.5) and +(-0.5, -0.5) .. (G-2);
%\draw[] (G-2) .. controls +(1, -1) and +(-1, -1) .. (G-7);
\draw[] (G-7) .. controls +(1, -1) and +(-1, 1) .. (G--9);
\draw[] (G--9) .. controls +(-0.5, 0.5) and +(0.5, 0.5) .. (G--8);
\draw[] (G--8) .. controls +(-1, 1) and +(1, 1) .. (G--2);
\draw[] (G--2) .. controls +(-0.75, 1) and +(0.75, -1) .. (G-1);
\draw[] (G-6) .. controls +(0.75, -1) and +(-0.75, 1) .. (G--7);
\draw[] (G--6) .. controls +(-0.75, 1) and +(0.75, -1) .. (G-5);
\draw[] (G-9) .. controls +(-1, -1) and +(1, 1) .. (G--6);
\draw[] (G--6) .. controls +(-0.6, 0.6) and +(0.6, 0.6) .. (G--4);
%\draw[] (G--4) .. controls +(0.75, 1) and +(-0.75, -1) .. (G-5);
\draw[] (G-4) .. controls +(0.5, -0.5) and +(-0.5, -0.5) .. (G-8);
%\draw[] (G-8) .. controls +(-1, -1) and +(1, 1) .. (G--5);
\draw[] (G--5) .. controls +(-0.6, 0.6) and +(0.6, 0.6) .. (G--3);
\draw[] (G--3) .. controls +(0.75, 1) and +(-0.75, -1) .. (G-4);
\draw[] (G-3) .. controls +(-1, -1) and +(1, 1) .. (G--1);
\node[vertex] (G--9p) at (12.0, -3.5) [shape = circle, draw] {};
\node[vertex] (G-5p) at (6.0, -1.5) [shape = circle, draw] {};
\node[vertex] (G--8p) at (10.5, -3.5) [shape = circle, draw] {};
\node[vertex] (G--7p) at (9.0, -3.5) [shape = circle, draw] {};
\node[vertex] (G--6p) at (7.5, -3.5) [shape = circle, draw] {};
\node[vertex] (G-7p) at (9.0, -1.5) [shape = circle, draw] {};
\node[vertex] (G-8p) at (10.5, -1.5) [shape = circle, draw] {};
\node[vertex] (G-9p) at (12.0, -1.5) [shape = circle, draw] {};
\node[vertex] (G--5p) at (6.0, -3.5) [shape = circle, draw] {};
\node[vertex] (G--3p) at (3.0, -3.5) [shape = circle, draw] {};
\node[vertex] (G-4p) at (4.5, -1.5) [shape = circle, draw] {};
\node[vertex] (G-6p) at (7.5, -1.5) [shape = circle, draw] {};
\node[vertex] (G--4p) at (4.5, -3.5) [shape = circle, draw] {};
\node[vertex] (G--1p) at (0.0, -3.5) [shape = circle, draw] {};
\node[vertex] (G-2p) at (1.5, -1.5) [shape = circle, draw] {};
\node[vertex] (G-3p) at (3.0, -1.5) [shape = circle, draw] {};
\node[vertex] (G--2p) at (1.5, -3.5) [shape = circle, draw] {};
\node[vertex] (G-1p) at (0.0, -1.5) [shape = circle, draw] {};
\draw[] (G-5p) .. controls +(1, -1) and +(-1, 1) .. (G--9p);
\draw[] (G-7p) .. controls +(0.5, -0.5) and +(-0.5, -0.5) .. (G-8p);
\draw[] (G-8p) .. controls +(0.5, -0.5) and +(-0.5, -0.5) .. (G-9p);
\draw[] (G-9p) .. controls +(-0.75, -1) and +(0.75, 1) .. (G--8p);
\draw[] (G--8p) .. controls +(-0.5, 0.5) and +(0.5, 0.5) .. (G--7p);
\draw[] (G--7p) .. controls +(-0.5, 0.5) and +(0.5, 0.5) .. (G--6p);
%\draw[] (G--6p) .. controls +(0.75, 1) and +(-0.75, -1) .. (G-7p);
\draw[] (G-4p) .. controls +(0.6, -0.6) and +(-0.6, -0.6) .. (G-6p);
%\draw[] (G-6p) .. controls +(-0.75, -1) and +(0.75, 1) .. (G--5p);
\draw[] (G--5p) .. controls +(-0.6, 0.6) and +(0.6, 0.6) .. (G--3p);
\draw[] (G-4p) .. controls +(0.75, -1) and +(-0.75, 1) .. (G--5p);
\draw[] (G-2p) .. controls +(0.5, -0.5) and +(-0.5, -0.5) .. (G-3p);
%\draw[] (G-3p) .. controls +(0.75, -1) and +(-0.75, 1) .. (G--4p);
\draw[] (G--4p) .. controls +(-0.7, 0.7) and +(0.7, 0.7) .. (G--1p);
\draw[] (G--1p) .. controls +(0.75, 1) and +(-0.75, -1) .. (G-2p);
\draw[] (G-1p) .. controls +(0.75, -1) and +(-0.75, 1) .. (G--2p);
        \draw[] (G--1) .. controls +(0, -.5) and +(0, +.5) .. (G-1p);
        \draw[] (G--2) .. controls +(0, -.5) and +(0, +.5) .. (G-2p);
        \draw[] (G--3) .. controls +(0, -.5) and +(0, +.5) .. (G-3p);
        \draw[] (G--4) .. controls +(0, -.5) and +(0, +.5) .. (G-4p);
        \draw[] (G--5) .. controls +(0, -.5) and +(0, +.5) .. (G-5p);
        \draw[] (G--6) .. controls +(0, -.5) and +(0, +.5) .. (G-6p);
        \draw[] (G--7) .. controls +(0, -.5) and +(0, +.5) .. (G-7p);
        \draw[] (G--8) .. controls +(0, -.5) and +(0, +.5) .. (G-8p);
        \draw[] (G--9) .. controls +(0, -.5) and +(0, +.5) .. (G-9p);
\end{tikzpicture}}
=
\begin{tikzpicture}[scale = 0.4,thick, baseline={(0,-1ex/2)}]
\tikzstyle{vertex} = [shape = circle, minimum size = 4pt, inner sep = 1pt]
\node[vertex] (G--9) at (12.0, -1) [shape = circle, draw] {};
\node[vertex] (G--8) at (10.5, -1) [shape = circle, draw] {};
\node[vertex] (G--7) at (9.0, -1) [shape = circle, draw] {};
\node[vertex] (G--6) at (7.5, -1) [shape = circle, draw] {};
\node[vertex] (G--4) at (4.5, -1) [shape = circle, draw] {};
\node[vertex] (G--1) at (0.0, -1) [shape = circle, draw] {};
\node[vertex] (G-1) at (0.0, 1) [shape = circle, draw] {};
\node[vertex] (G-2) at (1.5, 1) [shape = circle, draw] {};
\node[vertex] (G-4) at (4.5, 1) [shape = circle, draw] {};
\node[vertex] (G-6) at (7.5, 1) [shape = circle, draw] {};
\node[vertex] (G-7) at (9.0, 1) [shape = circle, draw] {};
\node[vertex] (G-8) at (10.5, 1) [shape = circle, draw] {};
\node[vertex] (G--5) at (6.0, -1) [shape = circle, draw] {};
\node[vertex] (G--3) at (3.0, -1) [shape = circle, draw] {};
\node[vertex] (G-5) at (6.0, 1) [shape = circle, draw] {};
\node[vertex] (G-9) at (12.0, 1) [shape = circle, draw] {};
\node[vertex] (G--2) at (1.5, -1) [shape = circle, draw] {};
\node[vertex] (G-3) at (3.0, 1) [shape = circle, draw] {};
\draw[] (G-1) .. controls +(0.5, -0.5) and +(-0.5, -0.5) .. (G-2);
\draw[] (G-2) .. controls +(0.6, -0.6) and +(-0.6, -0.6) .. (G-4);
\draw[] (G-4) .. controls +(0.6, -0.6) and +(-0.6, -0.6) .. (G-6);
\draw[] (G-6) .. controls +(0.5, -0.5) and +(-0.5, -0.5) .. (G-7);
\draw[] (G-7) .. controls +(0.5, -0.5) and +(-0.5, -0.5) .. (G-8);
%\draw[] (G-8) .. controls +(0.75, -1) and +(-0.75, 1) .. (G--9);
\draw[] (G--9) .. controls +(-0.5, 0.5) and +(0.5, 0.5) .. (G--8);
\draw[] (G--8) .. controls +(-0.5, 0.5) and +(0.5, 0.5) .. (G--7);
\draw[] (G--7) .. controls +(-0.5, 0.5) and +(0.5, 0.5) .. (G--6);
\draw[] (G--6) .. controls +(-0.6, 0.6) and +(0.6, 0.6) .. (G--4);
\draw[] (G--4) .. controls +(-0.7, 0.7) and +(0.7, 0.7) .. (G--1);
\draw[] (G--1) .. controls +(0, 1) and +(0, -1) .. (G-1);
\draw[] (G-5) .. controls +(0.8, -0.8) and +(-0.8, -0.8) .. (G-9);
\draw[] (G-5) .. controls +(0, -.5) and +(0, .5) .. (G--5);
\draw[] (G--5) .. controls +(-0.6, 0.6) and +(0.6, 0.6) .. (G--3);
%\draw[] (G--3) .. controls +(1, 1) and +(-1, -1) .. (G-5);
\draw[] (G-3) .. controls +(-0.75, -1) and +(0.75, 1) .. (G--2);
\end{tikzpicture}~.
\end{center}
\end{example}

This product is an extension of the product of permutations in the sense that
an element $\pi = \{A_1, \ldots, A_r\} \in \UU_k$ with all $|A_i|=2$ can be
identified with a permutation $\sigma \in \sym_k$ via
$\pi = \{ \{\sigma(1), \o{1}\}, \{\sigma(2), \o2\}, \ldots, \{ \sigma(k), \o{k}\} \}$.
Under this identification the identity permutation is the identity of the monoid,
the product for two such elements corresponds to composition of permutations, and
the maximal subgroup of $\UU_k$ containing the identity is $\sym_k$.

The set of idempotents of $\UU_k$ is denoted $E(\UU_k)$. The idempotents are all $\pi \in \UU_k$
such that $A \cap [k] = \overline{A} \cap [k]$ for all $A \in \pi$.
Each idempotent $\pi \in E(\UU_k)$ is therefore
determined by the set partition $\mathsf{top}(\pi)$ of $[k]$
and we use the notation $e_{\mathsf{top}(\pi)}$ to refer to this element.
We note the following properties of the idempotents (see for instance \cite[Section 2.5]{OSSZ.2022}):
\begin{enumerate}
    \item
        The product of two idempotents is
        \begin{equation*}
            e_{\pi} e_{\gamma} = e_{\pi \vee \gamma},
        \end{equation*}
        where $\vee$ is the join operation
        in the lattice of set partitions.

    \item
        $\sym_k$ acts on the set of idempotents $E(\UU_k)$ by conjugation:
        for all $\tau \in \sym_k$ and $e_{\pi} \in E(\UU_k)$,
        \begin{equation*}
            \tau e_{\pi} \tau^{-1} = e_{\tau^{-1}(\pi)}.
        \end{equation*}
        Consequently,
        \begin{equation*}
            e_{\pi} \tau = \tau e_{\tau(\pi)}
            \qquad\text{and}\qquad
            \tau e_{\pi} = e_{\tau^{-1}(\pi)} \tau.
        \end{equation*}

    \item
        Every element of $\UU_k$ is of the form $\tau e_{\pi}$
        for $\tau \in \sym_k$ and $\pi \vdash [k]$.

    \item \label{item:set partition same}
        If $\tau_1 e_{\pi} = \tau_2 e_{\gamma}$ for $\tau_1,\tau_2 \in \sym_k$, then $\pi = \gamma$.
\end{enumerate}

Let $x$ and $y$ be elements of a monoid $M$.
We say that $x$ and $y$ are \defn{$\mathscr{J}$-equivalent} (one of Green's relations
on semigroups \cite{Green.1951}) if $MxM = MyM$.
This is an equivalence relation and so partitions the elements of
$M$ into classes which are called the \defn{$\mathscr{J}$-classes} of $M$.
In the case that $M = \UU_k$, the
$\mathscr{J}$-classes are indexed by partitions $\mu$ of $k$
and are given by
$J_\mu = \{ \pi \in \UU_k : \mathsf{type}(\pi) = \mu \}$
(for details, see \cite[Proposition 3.5]{OSSZ.2022}).
By the above properties, we also have
$J_\mu = \{ \sigma e_\gamma \tau : \sigma, \tau \in \sym_k \}$
for any set partition $\gamma$ of $[k]$ such that $\mathsf{type}(\gamma) = \mu$.

%%%%%%%%%%%%%%%%%%%%%%%%%%%%%%%%%%%%%%%%%%%%%%%%
\section{A new partial order $\preceq$ on integer partitions}
\label{section:new partial order}
%%%%%%%%%%%%%%%%%%%%%%%%%%%%%%%%%%%%%%%%%%%%%%%%

We define a new partial order on integer partitions in terms of set partitions in Section~\ref{section.definition new order}
and relate it to refinement order on integer partitions in Section~\ref{section:comparison with refinement}.

%%%%%%%%%%%%%%%%%%%%%%%%%%%%%%%%%%%%%%%%%%%%%%%%
\subsection{Definition of a new partial order on integer partitions}
\label{section.definition new order}

\begin{definition}
    For $\mu, \lambda \vdash k$,
    define \defn{$\mu \preceq \lambda$} if there exist
    set partitions $\pi_0, \pi_1, \ldots, \pi_\ell \vdash [k]$
    of type $\lambda$ with $\pi_0 \vee \pi_1 \vee \cdots \vee \pi_\ell$
    of type $\mu$.
\end{definition}
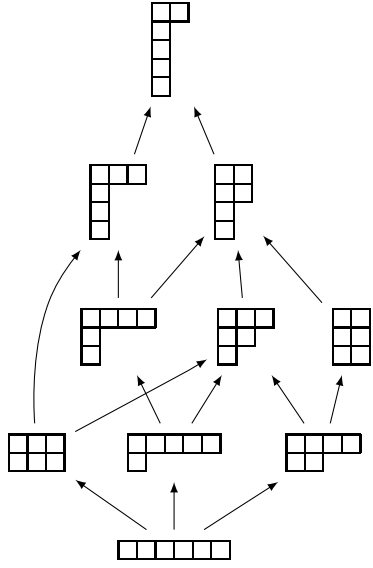
\begin{figure}
    \begin{tikzpicture}[>=latex,line join=bevel,xscale=0.75,yscale=0.6]
        \node (node_0) at (89.0bp,9.5bp) [draw,draw=none] {\ydiagram{6}};
        \node (node_1) at (89.0bp,70.5bp) [draw,draw=none] {\ydiagram{5,1}};
        \node (node_2) at (164.0bp,70.5bp) [draw,draw=none] {\ydiagram{4,2}};
        \node (node_3) at (61.0bp,143.5bp) [draw,draw=none] {\ydiagram{4,1,1}};
        \node (node_4) at (20.0bp,70.5bp) [draw,draw=none] {\ydiagram{3,3}};
        \node (node_5) at (125.0bp,143.5bp) [draw,draw=none] {\ydiagram{3,2,1}};
        \node (node_6) at (61.0bp,228.0bp) [draw,draw=none] {\ydiagram{3,1,1,1}};
        \node (node_7) at (178.0bp,143.5bp) [draw,draw=none] {\ydiagram{2,2,2}};
        \node (node_8) at (119.0bp,228.0bp) [draw,draw=none] {\ydiagram{2,2,1,1}};
        \node (node_9) at (87.0bp,324.0bp) [draw,draw=none] {\ydiagram{2,1,1,1,1}};
        \draw [black,->] (node_0) ..controls (89.0bp,25.933bp) and (89.0bp,35.596bp)  .. (node_1);
        \draw [black,->] (node_0) ..controls (109.83bp,26.884bp) and (124.45bp,38.385bp)  .. (node_2);
        \draw [black,->] (node_0) ..controls (69.925bp,26.811bp) and (56.644bp,38.167bp)  .. (node_4);
        \draw [black,->] (node_1) ..controls (80.21bp,93.788bp) and (76.453bp,103.31bp)  .. (node_3);
        \draw [black,->] (node_1) ..controls (100.38bp,93.949bp) and (105.34bp,103.72bp)  .. (node_5);
        \draw [black,->] (node_2) ..controls (151.67bp,93.949bp) and (146.3bp,103.72bp)  .. (node_5);
        \draw [black,->] (node_2) ..controls (168.36bp,93.627bp) and (170.19bp,102.91bp)  .. (node_7);
        \draw [black,->] (node_3) ..controls (61.0bp,173.07bp) and (61.0bp,182.11bp)  .. (node_6);
        \draw [black,->] (node_3) ..controls (82.574bp,175.19bp) and (90.979bp,187.14bp)  .. (node_8);
        \draw [black,->] (node_4) ..controls (55.489bp,94.736bp) and (77.184bp,109.1bp)  .. (96.0bp,122.0bp) .. controls (96.178bp,122.12bp) and (96.357bp,122.24bp)  .. (node_5);
        \draw [black,->] (node_4) ..controls (18.06bp,104.81bp) and (17.98bp,138.03bp)  .. (26.0bp,165.0bp) .. controls (28.716bp,174.14bp) and (32.972bp,183.39bp)  .. (node_6);
        \draw [black,->] (node_5) ..controls (122.92bp,173.07bp) and (122.26bp,182.11bp)  .. (node_8);
        \draw [black,->] (node_6) ..controls (70.47bp,263.24bp) and (72.979bp,272.31bp)  .. (node_9);
        \draw [black,->] (node_7) ..controls (156.25bp,174.91bp) and (147.5bp,187.15bp)  .. (node_8);
        \draw [black,->] (node_8) ..controls (107.32bp,263.32bp) and (104.19bp,272.5bp)  .. (node_9);
    \end{tikzpicture}
    \caption{Hasse diagram for $(\mathcal{P}_6 \setminus \{1^6\}, \preceq)$.}
\end{figure}

\begin{remark}
    Note that $1^k$ is incomparable to all other partitions:
    there is exactly one set partition $\pi$ of $[k]$ of type $1^k$,
    so $\mu \preceq 1^k$ implies $\mu = 1^k$; and
    if $\pi_0 \vee \pi_1 = \pi$,
    then $\pi_0 = \pi_1 = \pi$, so $1^k \preceq \mu$
    implies $\mu = 1^k$.
    For this reason we will mainly be concerned with the
    partial order on partitions of $k$ excluding $1^k$.
    We will make a connection between the submonoids of
    $\UU_k$ containing $\sym_k$ and the downsets of the
    partial order $(\mathcal{P}_k \setminus \{1^k\}, \preceq)$.
\end{remark}

\begin{prop}
    The relation $\preceq$ is a partial order on integer partitions of $k$.

    Moreover, if $\mu \preceq \lambda$, then $\mu$ is coarser than $\lambda$
    as an integer partition.
\end{prop}

\begin{proof}
    We first prove the second statement.
    Suppose $\mu \preceq \lambda$.
    Then $\mu = \mathsf{type}(\pi_{0} \vee \cdots \vee \pi_{\ell})$
    for some $\pi_{0}, \ldots, \pi_{\ell} \vdash [k]$ of type $\lambda$.
    Since $\pi_{0} \vee \cdots \vee \pi_{\ell}$ is obtained by
    merging blocks of $\pi_{0}$, its type is obtained
    by merging parts of $\mathsf{type}(\pi_{0}) = \lambda$. Hence,
    $\mu$ is coarser than $\lambda$.

    We now prove that $\preceq$ is a partial order.

    \smallskip \noindent
    \emph{Reflexivity}. $\lambda \preceq \lambda$ because there exists a set partition $\pi$ with
    $\mathsf{type}(\pi) = \lambda$.

    \smallskip \noindent
    \emph{Transitivity}.
    Suppose $\nu \preceq \mu$ and $\mu \preceq \lambda$.
    There exist
    \begin{itemize}
        \item
            $\pi_{0}, \ldots, \pi_{\ell} \vdash [k]$ of type $\lambda$
            with $\mathsf{type}(\pi_{0} \vee \cdots \vee \pi_{\ell}) = \mu$; and

        \item
            $\gamma_{0}, \ldots, \gamma_{m} \vdash [k]$ of type $\mu$
            with $\mathsf{type}(\gamma_{0} \vee \cdots \vee \gamma_{m}) = \nu$.
    \end{itemize}
    Since $\pi_{0} \vee \cdots \vee \pi_{\ell}$ has type $\mu$,
    there exist $u_0, \ldots, u_m \in \sym_k$ such that
    $\gamma_i
    = u_i(\pi_{0} \vee \cdots \vee \pi_{\ell})
    = u_i(\pi_{0}) \vee \cdots \vee u_i(\pi_{\ell})$ for all $0 \leqslant i \leqslant m$.
    Thus,
    \begin{equation*}
        \gamma_0 \vee \cdots \vee \gamma_m
        =
        u_0(\pi_{0}) \vee \cdots \vee u_0(\pi_{\ell})
        \vee \cdots \vee
        u_m(\pi_{0}) \vee \cdots \vee u_m(\pi_{\ell})
    \end{equation*}
    is a join of set partitions of type $\lambda$, showing that
    $\nu \preceq \lambda$.

    \smallskip \noindent
    \emph{Antisymmetry}.
    If $\mu \preceq \lambda$ and $\lambda \preceq \mu$, then
    $\mu$ is coarser than $\lambda$ and $\lambda$
    is coarser than $\mu$ in the refinement order
    on integer partitions and so $\mu = \lambda$.
\end{proof}

%%%%%%%%%%%%%%%%%%%%%%%%%%%%%%%%%%%%%%%%%%%%%%%%%%%%%%%%%%
\subsection{Comparison with refinement order}
\label{section:comparison with refinement}

The next result describes the precise relationship between
$\preceq$ and refinement order on integer partitions.

\begin{theorem}
    \label{smallest-part-reformulation}
    For $\nu \vdash k$ with $\nu \neq 1^k$, let $\smallest_\nu$ be the smallest
    part of $\nu$ not equal to $1$.
    Then
    \begin{equation*}
        \mu \preceq \lambda
        \qquad\text{if and only if}\qquad
        \mu \text{~is~coarser~than~} \lambda \text{~~and~~} \smallest_\mu \geqslant \smallest_\lambda.
    \end{equation*}
\end{theorem}

The proof is given below after the presentation of a few necessary lemmas.

Recall that if $\mu \lessdot \lambda$ is a cover relation in the refinement
order on integer partitions, then $\mu$ is obtained from $\lambda$
by merging two parts $\lambda_i$ and $\lambda_j$.
We first prove that $\mu \precdot \lambda$ is a cover relation
if at least one of the two merged parts is greater than $1$.

\begin{lemma}
    \label{merge-non-singleton-blocks}
    If $\mu$ is obtained from $\lambda$ by summing two parts
    $\lambda_i$ and $\lambda_j$ not both equal to $1$,
    then $\mu \precdot \lambda$.
\end{lemma}

\begin{proof}
    Let $\pi = \{B_1, \ldots, B_\ell\}$ be of type $\lambda$
    with $|B_k| = \lambda_k$ for all $k \in [\ell]$.
    Fix $b_i \in B_i$ and $b_j \in B_j$ and let
    \begin{gather*}
        \pi' = \big(\{B_1, \ldots, B_\ell\} \setminus \{B_i, B_j\}\big)
        \cup \big\{ B'_i, B'_j \big\},
        \shortintertext{where}
        B_i' = (B_i \setminus \{b_i\}) \cup \{b_j\}
        \quad\text{and}\quad
        B_j' = (B_j \setminus \{b_j\}) \cup \{b_i\}.
    \end{gather*}
    %Since $|B_i| = \lambda_i \geqslant 2$, the intersection
    %$B_i \cap B'_i$ is nonempty because it contains $B_i \setminus \{b_i\}$.
    %Also, $B_j \cap B'_i \neq \emptyset$ because it contains $b_j$.
    %Similarly,
    %$B_j \cap B'_j \neq \emptyset$ and $B_i \cap B'_j \neq \emptyset$
    %because $|B_j| = \lambda_j \geqslant 2$.
    %\Anne{What happens if one of $\lambda_i$ or $\lambda_j$ is equal to 1? You seem to be using that both are bigger than 1.}
    %Hence,
    Assume without loss of generality that $|B_i| = \lambda_i \geqslant 2$. Then the intersection
    $B_i \cap B'_i$ is nonempty because it contains $B_i \setminus \{b_i\}$.
    Also, $B_j \cap B'_i \neq \emptyset$ because it contains $b_j$. Furthermore, $\pi\neq \pi'$.
    Hence,
    \begin{equation*}
        \pi \vee \pi'
        = \Big(\{B_1, \ldots, B_\ell\} \setminus \{B_i, B_j\}\Big)
        \cup \Big\{ B_i \cup B_j \Big\}
    \end{equation*}
    is of type $\mu$ and so $\mu \prec \lambda$
    because $\mathsf{type}(\pi') = \mathsf{type}(\pi) = \lambda$.

    If $\mu \prec \lambda$ is not a cover relation,
    then $\mu \prec \nu \prec \lambda$ for some $\nu$,
    which implies $\mu$ is coarser than $\nu$ and $\nu$ is coarser than $\lambda$ in refinement
    order. But since $\mu \lessdot \lambda$ in refinement
    order, this is a contradiction. Thus, $\mu \precdot \lambda$.
\end{proof}

\begin{lemma}
    \label{merge-singleton-blocks}
    Suppose $\mu$ is obtained from $\lambda \neq 1^k$ by merging $s$ ones:
    explicitly,
    \begin{equation*}
        \lambda = (\lambda_1, \ldots, \lambda_{m}, 1^{t})
        \quad\text{and}\quad
        \mu = \mathsf{sort}(\lambda_1, \ldots, \lambda_{m}, s, 1^{t - s}),
    \end{equation*}
    where $\lambda_{m} > 1$ and $t \geqslant s$.
    If $s \geqslant \lambda_{m}$, then $\mu \prec \lambda$.
\end{lemma}

\begin{proof}
    Let $\pi = \{B_1, \ldots, B_\ell\}$ be a set partition of type $\lambda$
    with $|B_k| = \lambda_k$ for all $k \in [\ell]$.
    The block $B_{m}$ has size $\lambda_{m} \leqslant s$ and $B_{m+1}, \ldots, B_{\ell}$ are singletons,
    so we can write
    \begin{equation*}
        B_{m} = \{a_1, a_2, \ldots, a_{\lambda_{m}}\},
        \quad B_{m+1} = \{b_{1}\},
        \quad \ldots \quad , \quad
        B_{m+{\lambda_{m}}} = \{b_{{\lambda_{m}}}\}.
    \end{equation*}
    Let $\gamma$ be the set partition obtained from $\pi$ by replacing
    the set partitions
    $B_{m}, B_{m+1}, \ldots, B_{m+{\lambda_{m}}}$ by
    \begin{equation*}
        B'_{m} = \{b_1, b_2, \ldots, b_{\lambda_{m}}\},
        \quad B'_{m+1} = \{a_{1}\},
        \quad \ldots \quad , \quad
        B'_{m+{\lambda_{m}}} = \{a_{{\lambda_{m}}}\}.
    \end{equation*}
    Then $\gamma$ also has type $\lambda$ and
    \begin{equation*}
        \pi \vee \gamma =
        \pi \setminus \{B_{m+1}, \ldots, B_{m+\lambda_{m}}\} \cup \{B_{m+1} \cup \cdots \cup B_{m+\lambda_{m}}\}
    \end{equation*}
    has type
    $\nu := (\lambda_{1}, \ldots, \lambda_{m}, \lambda_{m}, 1^{t - \lambda_{m}})$,
    proving that $\nu \prec \lambda$.
    Note that $\mu \preceq \nu$ because
    $\mu$ can be obtained from $\nu$ by
    first merging $\lambda_m$ with $1$,
    then merging $\lambda_m + 1$ with another $1$, and so on:
    \begin{equation*}
        \begin{aligned}[b]
            \lambda \succ \nu & = (\lambda_{1}, \ldots, \lambda_{m}, \lambda_{m}, 1^{t - \lambda_{m}})
            \\
            & \succ \mathsf{sort}(\lambda_{1}, \ldots, \lambda_{m}, \lambda_{m} + 1, 1^{t - (\lambda_{m} + 1)})
            \\
            & \succ \mathsf{sort}(\lambda_{1}, \ldots, \lambda_{m}, \lambda_{m} + 2, 1^{t - (\lambda_{m} + 2)})
            \\
            & \succ \cdots
            \succ \mathsf{sort}(\lambda_{1}, \ldots, \lambda_{m}, s, 1^{t - s}) = \mu.
        \end{aligned}
        \qedhere
    \end{equation*}
\end{proof}

\begin{proof}[Proof of Theorem~\ref{smallest-part-reformulation}]
    ($\Leftarrow$)
    Suppose that $\mu$ is coarser than $\lambda$ and $\smallest_\mu \geqslant \smallest_\lambda$.

    If $\mu$ can be obtained from $\lambda$ by merging at each step with a part
    of size at least $2$, then $\mu \preceq \lambda$ by Lemma~\ref{merge-non-singleton-blocks}.

    Suppose instead that there are parts of $\mu$ that are obtained from
    $\lambda$ by merging $1$s.
    By first merging $\smallest_\lambda$ ones (Lemma~\ref{merge-singleton-blocks}),
    and then successively merging this new part with $1$s (Lemma~\ref{merge-non-singleton-blocks}),
    it follows that $\mu \preceq \lambda$.

    ($\Rightarrow$)
    Suppose $\mu \preceq \lambda$ and $\smallest_\mu < \smallest_\lambda$.
    Since $\mu \preceq \lambda$, there exist set partitions $\pi_0, \ldots,
    \pi_\ell$ of type $\lambda$ with $\pi_0 \vee \cdots \vee \pi_\ell$ of type $\mu$.
    This yields a chain in the lattice of set partitions
    \begin{equation*}
        \underbrace{\pi_0}_{\text{type~$\lambda$}}
        ~\leqslant~ \pi_0 \vee \pi_1
        ~\leqslant~ \cdots
        ~\leqslant~ \underbrace{\pi_0 \vee \pi_1 \vee \cdots \vee \pi_\ell}_{\text{type~$\mu$}}.
    \end{equation*}

    Since $\pi_0 \vee \cdots \vee \pi_\ell$
    contains a block of size $\smallest_\mu$
    and $\pi_0$ does not,
    there exists $i \in [\ell]$ such that
    $\gamma := \pi_0 \vee \cdots \vee \pi_{i-1}$
    does not contain a block of size $\smallest_\mu$
    and $\gamma \vee \pi_{i}$ does.

    This block of size $\smallest_\mu$ is necessarily obtained by merging
    singleton blocks of $\pi_{i}$ since all other blocks of $\pi_{i}$ are
    of size greater than $\smallest_\mu$.
    Hence,
    there exist blocks $C_1, \ldots, C_{\smallest_\mu-1} \in \gamma$
    and singletons $B_1, \ldots, B_{\smallest_\mu} \in \pi_i$
    such that
    $B_1 \sim_{C_1} B_2 \sim_{C_2} \cdots \sim_{C_{\smallest_\mu-1}} B_{\smallest_\mu}$,
    where $B \sim_{C} B'$ means $C$ intersects both $B$ and $B'$.

    Let us show that each $C_i$ equals $B_1 \cup \cdots \cup B_{\smallest_\mu}$.
    Since $B_1$ and $B_2$ are singletons and $B_1 \sim_{C_1} B_2$,
    we have that $B_1 \subseteq C_1$ and $B_2 \subseteq C_1$.
    Similarly, $B_2, B_3 \subseteq C_2$.
    But since $C_1 \cap C_2 = \emptyset$ if $C_1 \neq C_2$,
    it follows that $C_2 = C_1$ because both contain $B_2$.
    Continuing in this way, $C_{1} = \cdots = C_{\smallest_\mu-1}$,
    and $C_{1} \supseteq B_1 \cup \cdots \cup B_{\smallest_\mu}$.
    Since $C_{1}$ is contained in the block of $\gamma \vee \pi_i$
    that contains $B_1$, it follows that
    $C_{1} \subseteq B_1 \cup \cdots \cup B_{\smallest_\mu}$.

    Thus, $\gamma$ contains a block of size $\smallest_\mu$,
    contradicting the choice of $\gamma$.
\end{proof}

%%%%%%%%%%%%%%%%%%%%%%%%%%%%%%%%%%%%%%%%%%%%%%%%%%
\section{Cover relations for $\preceq$}
\label{section:cover relations}
%%%%%%%%%%%%%%%%%%%%%%%%%%%%%%%%%%%%%%%%%%%%%%%%%%

The goal of this section is the characterization of the cover relations of~$\preceq$.

\begin{theorem}
    \label{cover-relations}
    Let $\mu$ and $\lambda$ be two partitions of $k$. Then
    $\mu \precdot \lambda$ if and only if either:
    \begin{enumerate}
        \item
            $\mu$ is obtained from $\lambda$ by merging two parts
            $\lambda_i$ and $\lambda_j$ that are not both $1$;

        \item
            there exist integers $t \geqslant s \geqslant 2$ such that
            $\lambda_{m-1} \geqslant s$,
            \begin{equation*}
                \lambda = (\lambda_1, \ldots, \lambda_{m-1}, s, \underbrace{1, \ldots, 1}_{t})
                \quad\text{and}\quad
                \mu = (\lambda_1, \ldots, \lambda_{m-1}, s, s, \underbrace{1, \ldots, 1}_{t - s}).
            \end{equation*}
    \end{enumerate}
\end{theorem}

%%%%%%%%%%%%%%%%%%%%%%%%%%%%%%%%%%%%%%%%%%%%%%%%%
\subsection{Proof that the relations in Theorem~\ref{cover-relations} are cover relations}

It follows directly from Lemma~\ref{merge-non-singleton-blocks} that the first
form of the relations in Theorem~\ref{cover-relations} are cover relations.

Suppose there are integers $t \geqslant s \geqslant 2$ and $\lambda_{m-1} \geqslant s$ such that
\begin{equation*}
    \lambda = (\lambda_1, \ldots, \lambda_{m-1}, s, 1^t)
    \quad\text{and}\quad
    \mu = (\lambda_1, \ldots, \lambda_{m-1}, s, s, 1^{t - s}).
\end{equation*}
By Lemma~\ref{merge-singleton-blocks}, it follows that $\mu \prec \lambda$.
If this is not a cover relation, then there exists $\nu$ such that $\mu \prec \nu \prec \lambda$.
Then $\nu = (\lambda_1, \ldots, \lambda_{m-1}, s, r, 1^{t-r})$ with $1 < r < s$.
By Theorem~\ref{smallest-part-reformulation}, the smallest part of $\nu$ that is not equal to
$1$ (this is $r$) is greater than or equal to the smallest part of $\lambda$ that is not
equal to $1$ (which is $s$). This implies $r \geqslant s > r$, a contradiction.

%%%%%%%%%%%%%%%%%%%%%%%%%%%%%%%%%%%%%%%%%%%%%%%%%
\subsection{Proof that all cover relations for $\preceq$ are of the form specified in Theorem~\ref{cover-relations}}

We begin by introducing some notation.
For $\lambda = (\lambda_1, \ldots, \lambda_\ell)$ and $I \subseteq [\ell]$, define
\begin{equation*}
    \lambda_I = \sum_{i \in I} \lambda_i
    \qquad\text{and}\qquad
    \lambda^I = (\lambda_i : i \in I) \vdash \lambda_I.
\end{equation*}
If $\mu < \lambda$ in refinement order, then we can write
\begin{equation}
    \label{mu-from-lambda}
    \mu = \mathsf{sort}(\lambda_{I_1}, \lambda_{I_2}, \ldots, \lambda_{I_d})
    \quad\text{for some~}
    \{I_1, \ldots, I_d\} \vdash [\ell].
\end{equation}
Furthermore, we order the sets $I_1, \ldots, I_d$ so that
\begin{gather*}
    \max(\lambda^{I_1}) \geqslant \cdots \geqslant \max(\lambda^{I_d})
    \shortintertext{and}
    \lambda_{I_j} \geqslant \lambda_{I_{j+1}} \text{~if~} \max(\lambda^{I_j}) = \max(\lambda^{I_{j+1}}).
\end{gather*}
Note that this expression is not necessarily unique.\\

\begin{example}
    For $\lambda = (6,5,4,4,3,2,1,1,1)$
    and $\mu = (12, 12, 3)$, we can take
    \begin{gather*}
        \lambda^{\{1,3,6\}} = (6, 4, 2),
        \quad
        \lambda^{\{2,4,7,8,9\}} = (5, 4, 1, 1, 1),
        \quad
        \lambda^{\{5\}} = (3),
        \shortintertext{or}
        \lambda^{\{1,4,6\}} = (6, 4, 2),
        \quad
        \lambda^{\{2,3,5\}} = (5, 4, 3),
        \quad
        \lambda^{\{7,8,9\}} = (1, 1, 1).
    \end{gather*}
\end{example}

%\vskip .2in
%\paragraph{Proof that all cover relations for $\preceq$ are of the form specified in Theorem~\ref{cover-relations}}
We are now ready to prove that all cover relations for $\preceq$ are of the form specified in Theorem~\ref{cover-relations}.
Suppose $\mu \precdot \lambda$.
Then $\mu$ is coarser than $\lambda$ in refinement order, so there exists a
way of merging the blocks of $\lambda$ to obtain $\mu$,
as in \eqref{mu-from-lambda}.

\medskip
\noindent
\textbf{Case 1: Suppose every $\lambda^{I_j}$ in \eqref{mu-from-lambda} contains a part greater than $1$.}

\noindent
Construct a chain by merging parts of $\lambda$ as follows
(\cf Example~\ref{example-construct-chain-case-1}):
\begin{itemize}
    \item
        if $\lambda^{I_d} = (\lambda_{i_1} \geqslant \cdots \geqslant \lambda_{i_t})$,
        replace $\lambda_{i_1}$ and $\lambda_{i_2}$ by $\lambda_{i_1} + \lambda_{i_2}$;
    \item
        then replace $\lambda_{i_1} + \lambda_{i_2}$ and $\lambda_{i_3}$ by
        $\lambda_{i_1} + \lambda_{i_2} + \lambda_{i_3}$; and so on.
\end{itemize}
Proceed similarly with $I_{d-1}, \ldots, I_{1}$.
Since every step involves merging with a block of size greater than $1$,
we obtain a $\preceq$-chain from $\mu$ to $\lambda$
by Lemma~\ref{merge-non-singleton-blocks}.
Since $\mu \precdot \lambda$, it follows that $d = 1$ and
$\mu$ is obtained from $\lambda$ by merging two blocks that are not both $1$.

\begin{example}
    \label{example-construct-chain-case-1}
    Let $\lambda = (6,5,4,4,3,2,1,1,1)$ and $\mu = (12,12,3)$.
    For
    \begin{gather*}
        \lambda^{\{1,3,6\}} = (6, 4, 2),
        \quad
        \lambda^{\{2,4,7,8,9\}} = (5, 4, 1, 1, 1),
        \quad
        \lambda^{\{5\}} = (3),
    \end{gather*}
    one first merges $\lambda_2$ and $\lambda_4$;
    then $\lambda_2, \lambda_4$, and $\lambda_7$; and so on:
    \begin{equation*}
        \def\topnum#1#2{\overset{\text{\tiny\color{gray}#1}}{#2}}
        \begin{multlined}
            (\topnum{1}{6},\topnum{2}{5},\topnum{3}{4},\topnum{4}{4},\topnum{5}{3},\topnum{6}{2},\topnum{7}{1},\topnum{8}{1},\topnum{9}{1})
            \succ (\topnum{2,4}{9},\topnum{1}{6},\topnum{3}{4},\topnum{5}{3},\topnum{6}{2},\topnum{7}{1},\topnum{8}{1},\topnum{9}{1})
            \succ (\topnum{2,4,7}{10},\topnum{1}{6},\topnum{3}{4},\topnum{5}{3},\topnum{6}{2},\topnum{8}{1},\topnum{9}{1})
            \\
            \succ (\topnum{2,4,7,8}{11},\topnum{1}{6},\topnum{3}{4},\topnum{5}{3},\topnum{6}{2},\topnum{9}{1})
            \succ (\topnum{2,4,7,8,9}{12},\topnum{1}{6},\topnum{3}{4},\topnum{5}{3},\topnum{6}{2})
            \succ (\topnum{2,4,7,8,9}{12},\topnum{1,3}{10},\topnum{5}{3},\topnum{6}{2})
            \succ (12,12,3).
        \end{multlined}
    \end{equation*}
\end{example}

\medskip
\noindent
\textbf{Case 2: Suppose $\lambda^{I_d} = 1^r$ in \eqref{mu-from-lambda} with $r > 1$.}
By Theorem~\ref{smallest-part-reformulation},
\begin{equation*}
    r = |\lambda^{I_{d}}| \geqslant \smallest_\mu \geqslant \smallest_\lambda =: s.
\end{equation*}
We construct a $\preceq$-chain as follows.
By hypothesis,
$\lambda = (\lambda_1, \ldots, \lambda_{m-1}, s, 1^t)$
with $t \geqslant r$.
By Lemma~\ref{merge-singleton-blocks},
since $\lambda$ contains $s$, merging $s$ ones yields the following relation
\begin{equation*}
    \lambda = (\lambda_1, \ldots, \lambda_{m-1}, s, 1^t)
    \succ
    (\lambda_1, \ldots, \lambda_{m-1}, s, s, 1^{t - s}).
\end{equation*}
Since $s \geqslant 2$, merging $s$ with $1$ is
a $\preceq$-relation by Lemma~\ref{merge-non-singleton-blocks}:
\begin{equation*}
    (\lambda_1, \ldots, \lambda_{m-1}, s, s, 1^{t - s})
    \succ
    \mathsf{sort}(\lambda_1, \ldots, \lambda_{m-1}, s + 1, s, 1^{t - s - 1}).
\end{equation*}
Continuing to merge the new part with $1$s, we get
the chain
\begin{equation*}
    \begin{multlined}
        \mathsf{sort}(\lambda_1, \ldots, \lambda_{m-1}, s + 1, s, 1^{t - (s + 1)})
        \succ
        \mathsf{sort}(\lambda_1, \ldots, \lambda_{m-1}, s + 2, s, 1^{t - (s + 2)})
        \\
        \succ
        \cdots
        \succ
        \mathsf{sort}(\lambda_1, \ldots, \lambda_{m-1}, r, s, 1^{t - r})
        =: \nu^{(d)}.
    \end{multlined}
\end{equation*}

\begin{example}
    Consider $\lambda = (6,5,4,3,1,1,1,1)$, $\mu = (10,8,4)$,
    and
    \begin{gather*}
        \lambda^{I_1} = (6, 4),
        \quad
        \lambda^{I_2} = (5, 3),
        \quad
        \lambda^{I_3} = (1, 1, 1, 1).
    \end{gather*}
    One obtains the chain
    \begin{equation*}
        \begin{multlined}
            (6,5,4,3,\underline{1,1,1},1)
            \succ (6,5,4,3,\underline{3,1})
            \succ ({6},5,{4},4,3)
            =: \nu^{(3)}.
        \end{multlined}
    \end{equation*}
\end{example}

If $\lambda^{I_{d-1}}$ is also of the form $1^{r'}$, then $r' \geqslant r$
because $I_1, \ldots, I_d$ are ordered so that
$\lambda_{I_j} \geqslant \lambda_{I_{j+1}}$
when $\max(\lambda^{I_j}) = \max(\lambda^{I_{j+1}})$.
Therefore, we can apply the above procedure to $I_{d-1}$.
Continuing in this way, we obtain a chain
\begin{equation*}
    \nu^{(a)} \prec \nu^{(a+1)} \prec \cdots \prec \nu^{(d)}
    \prec (\lambda_1, \ldots, \lambda_{m-1}, s, s, 1^{t-s}) \precdot \lambda,
\end{equation*}
where $\nu^{(j)}$ is obtained from $\nu^{(j+1)}$ by merging the parts in $\lambda^{I_j}$;
and $a$ is such that $\max(\lambda^{I_{a}}) = 1$ and either $a = 1$ or $\max(\lambda^{I_{a-1}}) > 1$.

If $a = 1$, then $\nu^{(a)} = \mu$.
In the other case, proceed as in \emph{Case 1} to merge
the parts of $\lambda^{I_{a-1}}, \ldots, \lambda^{I_{1}}$
(i.e., starting with the maximal element in $\lambda^{I_{j}}$).
Thus, in both cases we obtain a chain for $\preceq$ of the form
\begin{equation*}
    \mu = \nu^{(1)} \prec \cdots \prec \nu^{(a)} \prec \cdots \prec \nu^{(d)}
    \prec (\lambda_1, \ldots, \lambda_{m-1}, s, s, 1^{t-s}) \prec \lambda.
\end{equation*}
Since $\mu \precdot \lambda$, this is a contradiction unless
$\mu = (\lambda_1, \ldots, \lambda_{m-1}, s, s, 1^{t-s})$.

%%%%%%%%%%%%%%%%%%%%%%%%%%%%%%%%%%%%%%%%%%%%%%%%%%%%
\section{The submonoids of $\UU_k$ containing $\sym_k$}
\label{section:submonoids}
%%%%%%%%%%%%%%%%%%%%%%%%%%%%%%%%%%%%%%%%%%%%%%%%%%%%

In this section, we prove that the lattice of submonoids of $\UU_k$ containing $\sym_k$ is distributive. 
We begin by showing that these submonoids are unions of $\J$-classes of $\UU_k$.

\begin{lemma}
    \label{lemma-union-of-j-classes}
    If $S$ is a submonoid of $\UU_k$ that contains $\sym_k$,
    then $S$ is a union of $\J$-classes of $\UU_k$.
\end{lemma}

\begin{proof}
    Suppose $S$ is a submonoid of $\UU_k$ that contains $\sym_k$,
    and let $J_{\lambda}$ be a $\J$-class of $\UU_k$.
    If $S \cap J_{\lambda} \neq \emptyset$, then
    there exist $u, v \in \sym_k$ and $\pi \vdash [k]$ of type
    $\lambda$ such that $u e_{\pi} v \in S$.
    Since $u^{-1}, v^{-1} \in \sym_k \subseteq S$, we have
    \begin{equation*}
        e_{\pi} = u^{-1} (u e_{\pi} v) v^{-1} \in S,
    \end{equation*}
    from which it follows that $\sigma e_{\pi} \tau \in S$
    for all $\sigma, \tau \in \sym_k$.
    Consequently,
    \begin{equation*}
        J_\lambda = \{ \sigma e_{\pi} \tau : \sigma, \tau \in \sym_k \}
        \subseteq S.
        \qedhere
    \end{equation*}
\end{proof}

\begin{prop}
    \label{unitary-submonoid-with-one-e-pi}
    Let $\langle \sym_k, e_{\pi} \rangle$ be the submonoid of $\UU_k$ generated by $\sym_k$ and $e_{\pi}$.
    \begin{enumerate}
        \item
            For $\gamma \vdash [k]$ with $\mathsf{type}(\gamma) \neq 1^k$, we have
            $e_{\gamma} \in \langle \sym_k, e_{\pi} \rangle$ if and only if $\mathsf{type}(\gamma) \preceq \mathsf{type}(\pi)$.

        \item The monoid
            $\langle \sym_k, e_{\pi} \rangle$ is the union of $\sym_k$ and $\J$-classes $J_{\mu}$
            with $\mu \preceq \mathsf{type}(\pi)$:
            \begin{equation*}
                \langle \sym_k, e_{\pi} \rangle = \sym_k \cup \bigcup_{\mu \preceq \mathsf{type}(\pi)} J_\mu.
            \end{equation*}
    \end{enumerate}
\end{prop}

\begin{proof}
    Every element of $\langle \sym_k, e_{\pi} \rangle$ is either in $\sym_k$ or is a product of the form
    \begin{align*}
        u_0 e_{\pi} u_1 e_{\pi} \cdots
        u_{\ell-2} e_{\pi} u_{\ell-1} e_{\pi} u_\ell e_{\pi} u_{\ell+1}
        \qquad\text{with $u_0, u_1, \ldots, u_{\ell+1} \in \sym_k$ for some $\ell \geqslant 0$}.
    \end{align*}
    Since $\mathsf{type}(\gamma) \neq 1^k$, then $e_\gamma$ is of this form.  Setting
    $v_{\ell+1} = u_{\ell+1}$ and $v_{i} = u_{i} v_{i+1}$ for $0 \leqslant i \leqslant \ell$
    so that $u_{i} = v_{i} v_{i + 1}^{-1}$,
    the above becomes
    \begin{equation*}
        \begin{aligned}
            & v_{0}
            (v_{1}^{-1} e_{\pi} v_{1})
            (v_{2}^{-1} e_{\pi} v_{2})
            \cdots
            (v_{\ell-1}^{-1} e_{\pi} v_{\ell-1})
            (v_{\ell}^{-1} e_{\pi} v_{\ell})
            (v_{\ell+1}^{-1} e_{\pi} v_{\ell+1})
            \\
            & = v_{0} \,
            e_{v_{1}(\pi)}
            e_{v_{2}(\pi)}
            \cdots
            e_{v_{\ell}(\pi)}
            e_{v_{\ell+1}(\pi)}
            \\
            & = v_{0} \,
            e_{v_{1}(\pi) \vee v_{2}(\pi) \vee \cdots \vee v_{\ell}(\pi) \vee v_{\ell+1}(\pi)}.
        \end{aligned}
    \end{equation*}
    If $e_\gamma = v_{0} \,e_{v_{1}(\pi) \vee v_{2}(\pi) \vee \cdots \vee v_{\ell}(\pi) \vee v_{\ell+1}(\pi)}$, then
    $\gamma = v_{1}(\pi) \vee \cdots \vee v_{\ell+1}(\pi)$ by item \eqref{item:set partition same} in Section~\ref{section:notation}.
    Hence $\gamma$ is a join of set partitions of type $\mathsf{type}(\pi)$;
    that is, $\mathsf{type}(\gamma) \preceq \mathsf{type}(\pi)$.

    Conversely, if $\mathsf{type}(\gamma) \preceq \mathsf{type}(\pi)$,
    then $\gamma = \pi_{1} \vee \cdots \vee \pi_{\ell+1}$ for some set partitions
    $\pi_{1}, \ldots, \pi_{\ell+1}$ with $\mathsf{type}(\pi_{i}) = \mathsf{type}(\pi)$.
    Each $\pi_{i}$ is equal to $v_i(\pi)$ for some $v_i \in \sym_k$,
    and so $\gamma = v_{1}(\pi) \vee \cdots \vee v_{\ell+1}(\pi)$.
    Thus,
    \begin{equation*}
        e_\gamma = e_{v_{1}(\pi) \vee \cdots \vee v_{\ell+1}(\pi)}
        = v_{1} e_{\pi} v_{1}^{-1} \cdots v_{\ell+1} e_{\pi} v_{\ell+1}^{-1} \in \langle \sym_k, e_{\pi} \rangle.
        \qedhere
    \end{equation*}
\end{proof}

\begin{theorem}
    \label{union-is-submonoid}
    Let $S_1$ and $S_2$ be two submonoids of $\UU_k$ that contain $\sym_k$.
    Then $S_1 \cup  S_2$ is also a submonoid of $\UU_k$.
\end{theorem}

\begin{proof}
    Let $\tau e_{\pi} \in S_1$ and $\sigma e_{\gamma} \in S_2$.
    Then
    \begin{equation*}
        \tau e_{\pi} \cdot \sigma e_{\gamma}
        = \tau \sigma e_{\sigma(\pi)} e_{\gamma}
        = (\tau \sigma) e_{\sigma(\pi) \vee \gamma}.
    \end{equation*}
    We will show that $e_{\sigma(\pi) \vee \gamma}$ is contained in
    $S_1$ or $S_2$.

    Let
    \begin{gather*}
        \mu := \mathsf{type}(\sigma(\pi) \vee \gamma),
        \quad \lambda := \mathsf{type}(\pi),
        \quad \nu := \mathsf{type}(\gamma).
    \end{gather*}
    If $\lambda = 1^k$,
    then $\sigma(\pi) \vee \gamma = \gamma$
    because $\sigma(\pi) = \pi$ is the unique set partition of type $1^k$,
    from which it follows that $\tau \sigma e_{\sigma(\pi) \vee \gamma} \in S_2$.
    Similarly, if $\nu = 1^k$, then $\sigma(\pi) \vee \gamma = \sigma(\pi)$ and
    $\tau e_{\pi} \sigma e_{\sigma(\pi)} \in S_1$.

    Suppose $\mu, \lambda \neq 1^k$.
    Since we want to show that
    $e_{\sigma(\pi) \vee \gamma} \in S_1 \cup S_2$,
    it suffices to prove
    $e_{\sigma(\pi) \vee \gamma} \in
    \langle \sym_k, e_{\pi} \rangle \cup \langle \sym_k, e_{\gamma} \rangle$
    because
    $\langle \sym_k, e_{\pi} \rangle \subseteq S_1$
    and $\langle \sym_k, e_{\gamma} \rangle \subseteq S_2$.
    By Proposition~\ref{unitary-submonoid-with-one-e-pi},
    $e_{\sigma(\pi) \vee \gamma} \in \langle \sym_k, e_{\pi} \rangle$ if and only if $\mu \preceq \lambda$
    and
    $e_{\sigma(\pi) \vee \gamma} \in \langle \sym_k, e_{\gamma} \rangle$ if and only if $\mu \preceq \nu$,
    so it suffices to show
    $\mu \preceq \lambda$ or $\mu \preceq \nu$.

    Suppose $\mu \not\preceq \lambda$ and $\mu \not\preceq \nu$.
    Note that $\mu$ is coarser than $\lambda$ and $\mu$ is coarser
    than $\nu$ because $\sigma(\pi) \vee \gamma$
    is obtained by merging blocks of $\sigma(\pi)$ and by merging
    blocks of $\gamma$.
    Then by Theorem~\ref{smallest-part-reformulation}
    we have that
    $\smallest_\mu < \smallest_\nu$
    and
    $\smallest_\mu < \smallest_\lambda$.

    Since $\sigma(\pi) \vee \gamma$ has type $\mu$
    and is obtained by merging blocks in $\sigma(\pi)$,
    it follows that it contains a block of size $\smallest_\mu$ that can only
    be obtained by merging singleton blocks in $\sigma(\pi)$.
    This implies that $\gamma$ contains a block of size $\smallest_\mu$
    (\cf the proof of Theorem~\ref{smallest-part-reformulation}).
    Since $\gamma$ contains a block of size $\smallest_\mu$, it follows
    that $\nu$ has a part equal to $\smallest_\mu$,
    implying that $\smallest_\mu \geqslant \smallest_\nu$, a contradiction
    because $\smallest_\mu < \smallest_\nu$.
\end{proof}

\begin{cor}
\label{cor:submonoid characterization}
    Let $S$ be a submonoid of $\UU_k$ containing $\sym_k$
    and let $E(S)$ be its set of idempotents.
    Then
    $I := \{\mathsf{type}(\pi) : e_{\pi} \in E(S), e_\pi \neq id \}$
    is a down set for the partial order $\preceq$
    and
    $S = \sym_k \cup \bigcup_{\lambda \in I} J_\lambda$.
    
    Conversely, if $I$ is a downset of $(\mathcal{P}_k \backslash \{1^k\}, \preceq)$,
    then $\sym_k \cup \bigcup_{\lambda \in I} J_\lambda$ is a submonoid of $\UU_k$.
\end{cor}

\begin{proof}
    Let us first prove that $I$ is a down set.
    Suppose $\mu \preceq \mathsf{type}(\pi)$ for some $\pi \in E(S)$.
    By Proposition~\ref{unitary-submonoid-with-one-e-pi},
    $e_{\gamma} \in \langle \sym_k, e_{\pi} \rangle$ for any $\gamma \vdash [k]$
    of type $\mu$.
    But if $e_{\gamma} \in \langle \sym_k, e_{\pi} \rangle$,
    then $e_{\gamma} \in E(S)$, and so $\mu = \mathsf{type}(\gamma) \in I$.

    Every element of $S$ is of the form $\sigma e_{\pi}$ with $\sigma \in \sym_k$
    and $e_{\pi} \in E(S)$,
    so $S = \langle \sym_{k}, E(S) \rangle$.
    By Theorem~\ref{union-is-submonoid},
    $\langle \sym_{k}, E(S) \rangle = \bigcup_{e_{\pi} \in E(S)} \langle \sym_k, e_{\pi} \rangle$,
    since both are the smallest submonoids containing $\sym_k$ and $E(S)$.
    By Proposition~\ref{unitary-submonoid-with-one-e-pi},
    \begin{equation*}
        S = \bigcup_{e_{\pi} \in E(S)} \langle \sym_k, e_{\pi} \rangle
        = \bigcup_{e_{\pi} \in E(S)} \bigg(\sym_k \cup \bigcup_{\lambda \preceq \mathsf{type}(\pi)} J_{\lambda}\bigg)
        = \sym_k \cup \bigcup_{\lambda \in I} J_{\lambda}.
    \end{equation*}
    
    Now conversely, if $I$ is a downset of $(\mathcal{P}_k \backslash \{1^k\}, \preceq)$, then by
    Proposition \ref{unitary-submonoid-with-one-e-pi} part (2),
    for each $e_\pi$ with $\mathsf{type}(\pi) \in I$,
    $\left< \sym_k, e_\pi \right> \subseteq \sym_k \cup \bigcup_{\lambda \in I} J_\lambda$ and hence
    $$\sym_k \cup \bigcup_{\lambda \in I} J_\lambda = \bigcup_{\mathsf{type}(\pi) \in I} \left< \sym_k, e_\pi \right>~.$$
    Then by Theorem \ref{union-is-submonoid}, $\sym_k \cup \bigcup_{\lambda \in I} J_\lambda$ is a submonoid of $\UU_k$.
\end{proof}

Theorem \ref{union-is-submonoid} implies that the poset of submonoids of $\UU_k$
containing $\sym_k$ under union and intersection is a distributive lattice.
The lattices for $k=4,5,6$ are presented in Figure \ref{fig:posets}.
\begin{figure}
\begin{raisebox}{.75in}{
\begin{scalebox}{.6}{
\begin{tikzpicture}[>=latex,line join=bevel,]
\node (node_0) at (45.5bp,8.5bp) [draw,draw=none] {\begin{tabular}{c}\fbox{$\emptyset$}\\24\end{tabular}};
  \node (node_1) at (45.5bp,61.5bp) [draw,draw=none] {\begin{tabular}{c}\fbox{$4$}\\25\end{tabular}};
  \node (node_2) at (18.5bp,114.5bp) [draw,draw=none] {\begin{tabular}{c}\fbox{$22$}\\43\end{tabular}};
  \node (node_3) at (73.5bp,114.5bp) [draw,draw=none] {\begin{tabular}{c}\fbox{$31$}\\41\end{tabular}};
  \node (node_4) at (45.5bp,167.5bp) [draw,draw=none] {\begin{tabular}{c}\fbox{$31, 22$}\\59\end{tabular}};
  \node (node_5) at (45.5bp,220.5bp) [draw,draw=none] {\begin{tabular}{c}\fbox{$211$}\\131\end{tabular}};
  \draw [black,->] (node_0) ..controls (45.5bp,23.805bp) and (45.5bp,34.034bp)  .. (node_1);
  \draw [black,->] (node_1) ..controls (37.767bp,77.106bp) and (32.026bp,87.95bp)  .. (node_2);
  \draw [black,->] (node_1) ..controls (53.519bp,77.106bp) and (59.473bp,87.95bp)  .. (node_3);
  \draw [black,->] (node_2) ..controls (26.233bp,130.11bp) and (31.974bp,140.95bp)  .. (node_4);
  \draw [black,->] (node_3) ..controls (65.481bp,130.11bp) and (59.527bp,140.95bp)  .. (node_4);
  \draw [black,->] (node_4) ..controls (45.5bp,182.81bp) and (45.5bp,193.03bp)  .. (node_5);
\end{tikzpicture}
}
\end{scalebox}}
\end{raisebox}
\hskip .1in
\begin{raisebox}{.55in}{
\begin{scalebox}{.5}{
\begin{tikzpicture}[>=latex,line join=bevel,]
\node (node_0) at (59.5bp,8.5bp) [draw,draw=none]  {\begin{tabular}{c}\fbox{$\emptyset$}\\120\end{tabular}};
  \node (node_1) at (59.5bp,61.5bp) [draw,draw=none]  {\begin{tabular}{c}\fbox{$5$}\\121\end{tabular}};
  \node (node_2) at (32.5bp,114.5bp) [draw,draw=none]  {\begin{tabular}{c}\fbox{$32$}\\221\end{tabular}};
  \node (node_3) at (87.5bp,114.5bp) [draw,draw=none]  {\begin{tabular}{c}\fbox{$41$}\\146\end{tabular}};
  \node (node_5) at (32.5bp,167.5bp) [draw,draw=none]  {\begin{tabular}{c}\fbox{$41, 32$}\\246\end{tabular}};
  \node (node_4) at (106.5bp,167.5bp) [draw,draw=none]  {\begin{tabular}{c}\fbox{$311$}\\346\end{tabular}};
  \node (node_7) at (106.5bp,220.5bp) [draw,draw=none]  {\begin{tabular}{c}\fbox{$32, 311$}\\446\end{tabular}};
  \node (node_6) at (28.5bp,220.5bp) [draw,draw=none]  {\begin{tabular}{c}\fbox{$221$}\\696\end{tabular}};
  \node (node_8) at (67.5bp,273.5bp) [draw,draw=none]  {\begin{tabular}{c}\fbox{$311, 221$}\\896\end{tabular}};
  \node (node_9) at (67.5bp,326.5bp) [draw,draw=none]  {\begin{tabular}{c}\fbox{$2111$}\\1496\end{tabular}};
  \draw [black,->] (node_0) ..controls (59.5bp,23.805bp) and (59.5bp,34.034bp)  .. (node_1);
  \draw [black,->] (node_1) ..controls (51.767bp,77.106bp) and (46.026bp,87.95bp)  .. (node_2);
  \draw [black,->] (node_1) ..controls (67.519bp,77.106bp) and (73.473bp,87.95bp)  .. (node_3);
  \draw [black,->] (node_2) ..controls (32.5bp,129.81bp) and (32.5bp,140.03bp)  .. (node_5);
  \draw [black,->] (node_3) ..controls (92.885bp,129.96bp) and (96.81bp,140.49bp)  .. (node_4);
  \draw [black,->] (node_3) ..controls (71.179bp,130.63bp) and (58.29bp,142.59bp)  .. (node_5);
  \draw [black,->] (node_4) ..controls (106.5bp,182.81bp) and (106.5bp,193.03bp)  .. (node_7);
  \draw [black,->] (node_5) ..controls (31.378bp,182.81bp) and (30.576bp,193.03bp)  .. (node_6);
  \draw [black,->] (node_5) ..controls (54.896bp,183.93bp) and (73.185bp,196.54bp)  .. (node_7);
  \draw [black,->] (node_6) ..controls (39.842bp,236.33bp) and (48.494bp,247.65bp)  .. (node_8);
  \draw [black,->] (node_7) ..controls (95.158bp,236.33bp) and (86.506bp,247.65bp)  .. (node_8);
  \draw [black,->] (node_8) ..controls (67.5bp,288.81bp) and (67.5bp,299.03bp)  .. (node_9);
\end{tikzpicture}}
\end{scalebox}}
\end{raisebox}
\hskip .01in
\begin{scalebox}{.45}{
\begin{tikzpicture}[>=latex,line join=bevel,]
\node (node_0) at (245.0bp,8.5bp) [draw,draw=none]  {\begin{tabular}{c}\fbox{$\emptyset$}\\720\end{tabular}};
  \node (node_1) at (245.0bp,61.5bp) [draw,draw=none]  {\begin{tabular}{c}\fbox{$6$}\\721\end{tabular}};
  \node (node_2) at (245.0bp,114.5bp) [draw,draw=none]  {\begin{tabular}{c}\fbox{$33$}\\921\end{tabular}};
  \node (node_3) at (176.0bp,114.5bp) [draw,draw=none]  {\begin{tabular}{c}\fbox{$42$}\\946\end{tabular}};
  \node (node_7) at (307.0bp,114.5bp) [draw,draw=none]  {\begin{tabular}{c}\fbox{$51$}\\757\end{tabular}};
  \node (node_5) at (148.0bp,167.5bp) [draw,draw=none]  {\begin{tabular}{c}\fbox{$42, 33$}\\1146\end{tabular}};
  \node (node_9) at (314.0bp,167.5bp) [draw,draw=none]  {\begin{tabular}{c}\fbox{$51, 33$}\\957\end{tabular}};
  \node (node_4) at (74.0bp,167.5bp) [draw,draw=none]  {\begin{tabular}{c}\fbox{$222$}\\2296\end{tabular}};
  \node (node_12) at (231.0bp,167.5bp) [draw,draw=none]  {\begin{tabular}{c}\fbox{$51, 42$}\\982\end{tabular}};
  \node (node_6) at (37.0bp,220.5bp) [draw,draw=none]  {\begin{tabular}{c}\fbox{$33, 222$}\\2496\end{tabular}};
  \node (node_14) at (129.0bp,220.5bp) [draw,draw=none]  {\begin{tabular}{c}\fbox{$51, 222$}\\2332\end{tabular}};
  \node (node_16) at (231.0bp,220.5bp) [draw,draw=none]  {\begin{tabular}{c}\fbox{$51, 42, 33$}\\1182\end{tabular}};
  \node (node_22) at (88.0bp,273.5bp) [draw,draw=none]  {\begin{tabular}{c}\fbox{$51, 33, 222$}\\2532\end{tabular}};
  \node (node_8) at (388.0bp,167.5bp) [draw,draw=none]  {\begin{tabular}{c}\fbox{$411$}\\1207\end{tabular}};
  \node (node_10) at (425.0bp,220.5bp) [draw,draw=none]  {\begin{tabular}{c}\fbox{$411, 33$}\\1407\end{tabular}};
  \node (node_13) at (333.0bp,220.5bp) [draw,draw=none]  {\begin{tabular}{c}\fbox{$42, 411$}\\1432\end{tabular}};
  \node (node_11) at (471.0bp,273.5bp) [draw,draw=none]  {\begin{tabular}{c}\fbox{$3111$}\\3807\end{tabular}};
  \node (node_18) at (374.0bp,273.5bp) [draw,draw=none]  {\begin{tabular}{c}\fbox{$42, 411, 33$}\\1632\end{tabular}};
  \node (node_20) at (471.0bp,326.5bp) [draw,draw=none]  {\begin{tabular}{c}\fbox{$42, 3111$}\\4032\end{tabular}};
  \node (node_15) at (263.0bp,273.5bp) [draw,draw=none]  {\begin{tabular}{c}\fbox{$411, 222$}\\2782\end{tabular}};
  \node (node_24) at (253.0bp,326.5bp) [draw,draw=none]  {\begin{tabular}{c}\fbox{$411, 33, 222$}\\2982\end{tabular}};
  \node (node_17) at (180.0bp,273.5bp) [draw,draw=none]  {\begin{tabular}{c}\fbox{$321$}\\4782\end{tabular}};
  \node (node_19) at (369.0bp,326.5bp) [draw,draw=none]  {\begin{tabular}{c}\fbox{$411, 321$}\\5232\end{tabular}};
  \node (node_23) at (137.0bp,326.5bp) [draw,draw=none]  {\begin{tabular}{c}\fbox{$321, 222$}\\6132\end{tabular}};
  \node (node_21) at (471.0bp,379.5bp) [draw,draw=none]  {\begin{tabular}{c}\fbox{$321, 3111$}\\7632\end{tabular}};
  \node (node_26) at (235.0bp,379.5bp) [draw,draw=none]  {\begin{tabular}{c}\fbox{$411, 321, 222$}\\6582\end{tabular}};
  \node (node_25) at (360.0bp,379.5bp) [draw,draw=none]  {\begin{tabular}{c}\fbox{$3111, 222$}\\5382\end{tabular}};
  \node (node_28) at (356.0bp,432.5bp) [draw,draw=none]  {\begin{tabular}{c}\fbox{$321, 3111, 222$}\\8982\end{tabular}};
  \node (node_27) at (242.0bp,432.5bp) [draw,draw=none]  {\begin{tabular}{c}\fbox{$2211$}\\14682\end{tabular}};
  \node (node_29) at (297.0bp,485.5bp) [draw,draw=none]  {\begin{tabular}{c}\fbox{$3111, 2211$}\\17082\end{tabular}};
  \node (node_30) at (297.0bp,538.5bp) [draw,draw=none]  {\begin{tabular}{c}\fbox{$21111$}\\22482\end{tabular}};
  \draw [black,->] (node_0) ..controls (245.0bp,23.805bp) and (245.0bp,34.034bp)  .. (node_1);
  \draw [black,->] (node_1) ..controls (245.0bp,76.805bp) and (245.0bp,87.034bp)  .. (node_2);
  \draw [black,->] (node_1) ..controls (224.22bp,77.86bp) and (207.39bp,90.3bp)  .. (node_3);
  \draw [black,->] (node_1) ..controls (263.49bp,77.709bp) and (278.22bp,89.823bp)  .. (node_7);
  \draw [black,->] (node_2) ..controls (215.07bp,131.24bp) and (189.83bp,144.51bp)  .. (node_5);
  \draw [black,->] (node_2) ..controls (265.78bp,130.86bp) and (282.61bp,143.3bp)  .. (node_9);
  \draw [black,->] (node_3) ..controls (144.45bp,131.27bp) and (117.74bp,144.63bp)  .. (node_4);
  \draw [black,->] (node_3) ..controls (167.98bp,130.11bp) and (162.03bp,140.95bp)  .. (node_5);
  \draw [black,->] (node_3) ..controls (192.32bp,130.63bp) and (205.21bp,142.59bp)  .. (node_12);
  \draw [black,->] (node_4) ..controls (63.294bp,183.26bp) and (55.2bp,194.41bp)  .. (node_6);
  \draw [black,->] (node_4) ..controls (90.321bp,183.63bp) and (103.21bp,195.59bp)  .. (node_14);
  \draw [black,->] (node_5) ..controls (113.42bp,184.39bp) and (83.804bp,198.0bp)  .. (node_6);
  \draw [black,->] (node_5) ..controls (173.36bp,184.09bp) and (194.42bp,197.02bp)  .. (node_16);
  \draw [black,->] (node_6) ..controls (52.058bp,236.56bp) and (63.85bp,248.35bp)  .. (node_22);
  \draw [black,->] (node_7) ..controls (331.63bp,131.01bp) and (351.92bp,143.78bp)  .. (node_8);
  \draw [black,->] (node_7) ..controls (308.96bp,129.81bp) and (310.37bp,140.03bp)  .. (node_9);
  \draw [black,->] (node_7) ..controls (284.0bp,130.93bp) and (265.22bp,143.54bp)  .. (node_12);
  \draw [black,->] (node_8) ..controls (398.71bp,183.26bp) and (406.8bp,194.41bp)  .. (node_10);
  \draw [black,->] (node_8) ..controls (371.68bp,183.63bp) and (358.79bp,195.59bp)  .. (node_13);
  \draw [black,->] (node_9) ..controls (348.58bp,184.39bp) and (378.2bp,198.0bp)  .. (node_10);
  \draw [black,->] (node_9) ..controls (288.64bp,184.09bp) and (267.58bp,197.02bp)  .. (node_16);
  \draw [black,->] (node_10) ..controls (438.51bp,236.48bp) and (449.01bp,248.11bp)  .. (node_11);
  \draw [black,->] (node_10) ..controls (409.94bp,236.56bp) and (398.15bp,248.35bp)  .. (node_18);
  \draw [black,->] (node_11) ..controls (471.0bp,288.81bp) and (471.0bp,299.03bp)  .. (node_20);
  \draw [black,->] (node_12) ..controls (262.55bp,184.27bp) and (289.26bp,197.63bp)  .. (node_13);
  \draw [black,->] (node_12) ..controls (199.45bp,184.27bp) and (172.74bp,197.63bp)  .. (node_14);
  \draw [black,->] (node_12) ..controls (231.0bp,182.81bp) and (231.0bp,193.03bp)  .. (node_16);
  \draw [black,->] (node_13) ..controls (311.92bp,236.86bp) and (294.84bp,249.3bp)  .. (node_15);
  \draw [black,->] (node_13) ..controls (344.92bp,236.33bp) and (354.02bp,247.65bp)  .. (node_18);
  \draw [black,->] (node_14) ..controls (171.14bp,237.54bp) and (207.79bp,251.49bp)  .. (node_15);
  \draw [black,->] (node_14) ..controls (117.08bp,236.33bp) and (107.98bp,247.65bp)  .. (node_22);
  \draw [black,->] (node_15) ..controls (260.2bp,288.81bp) and (258.19bp,299.03bp)  .. (node_24);
  \draw [black,->] (node_16) ..controls (215.94bp,236.56bp) and (204.15bp,248.35bp)  .. (node_17);
  \draw [black,->] (node_16) ..controls (276.18bp,237.61bp) and (315.78bp,251.73bp)  .. (node_18);
  \draw [black,->] (node_16) ..controls (185.82bp,237.61bp) and (146.22bp,251.73bp)  .. (node_22);
  \draw [black,->] (node_17) ..controls (235.12bp,289.37bp) and (293.05bp,305.01bp)  .. (node_19);
  \draw [black,->] (node_17) ..controls (167.43bp,289.41bp) and (157.76bp,300.88bp)  .. (node_23);
  \draw [black,->] (node_18) ..controls (372.6bp,288.81bp) and (371.59bp,299.03bp)  .. (node_19);
  \draw [black,->] (node_18) ..controls (403.93bp,290.24bp) and (429.17bp,303.51bp)  .. (node_20);
  \draw [black,->] (node_18) ..controls (336.13bp,290.46bp) and (303.44bp,304.24bp)  .. (node_24);
  \draw [black,->] (node_19) ..controls (400.55bp,343.27bp) and (427.26bp,356.63bp)  .. (node_21);
  \draw [black,->] (node_19) ..controls (326.86bp,343.54bp) and (290.21bp,357.49bp)  .. (node_26);
  \draw [black,->] (node_20) ..controls (471.0bp,341.81bp) and (471.0bp,352.03bp)  .. (node_21);
  \draw [black,->] (node_20) ..controls (436.42bp,343.39bp) and (406.8bp,357.0bp)  .. (node_25);
  \draw [black,->] (node_21) ..controls (435.18bp,396.39bp) and (404.49bp,410.0bp)  .. (node_28);
  \draw [black,->] (node_22) ..controls (102.4bp,289.48bp) and (113.57bp,301.11bp)  .. (node_23);
  \draw [black,->] (node_22) ..controls (140.5bp,290.73bp) and (187.02bp,305.11bp)  .. (node_24);
  \draw [black,->] (node_23) ..controls (167.24bp,343.24bp) and (192.74bp,356.51bp)  .. (node_26);
  \draw [black,->] (node_24) ..controls (286.25bp,343.35bp) and (314.63bp,356.87bp)  .. (node_25);
  \draw [black,->] (node_24) ..controls (247.9bp,341.96bp) and (244.18bp,352.49bp)  .. (node_26);
  \draw [black,->] (node_25) ..controls (358.88bp,394.81bp) and (358.08bp,405.03bp)  .. (node_28);
  \draw [black,->] (node_26) ..controls (236.96bp,394.81bp) and (238.37bp,405.03bp)  .. (node_27);
  \draw [black,->] (node_26) ..controls (272.87bp,396.46bp) and (305.56bp,410.24bp)  .. (node_28);
  \draw [black,->] (node_27) ..controls (258.32bp,448.63bp) and (271.21bp,460.59bp)  .. (node_29);
  \draw [black,->] (node_28) ..controls (338.41bp,448.71bp) and (324.39bp,460.82bp)  .. (node_29);
  \draw [black,->] (node_29) ..controls (297.0bp,500.81bp) and (297.0bp,511.03bp)  .. (node_30);
\end{tikzpicture}}
\end{scalebox}
\caption{The distributive lattices of submonoids of $\UU_k$ containing the symmetric group $\sym_k$
for $k=4,5,6$.  Each submonoid is labeled by a box containing an antichain of the poset $(\mathcal{P}_k \backslash \{1^k\},\preceq)$.
The antichains of $\preceq$ are in bijection with the downsets of this poset
(see \cite[Section 3.4]{Stanley.ECI})
and in turn are in bijection with the submonoids by Corollary~\ref{cor:submonoid characterization}.
Below each box is the size of the corresponding submonoid.\label{fig:posets}}
\end{figure}
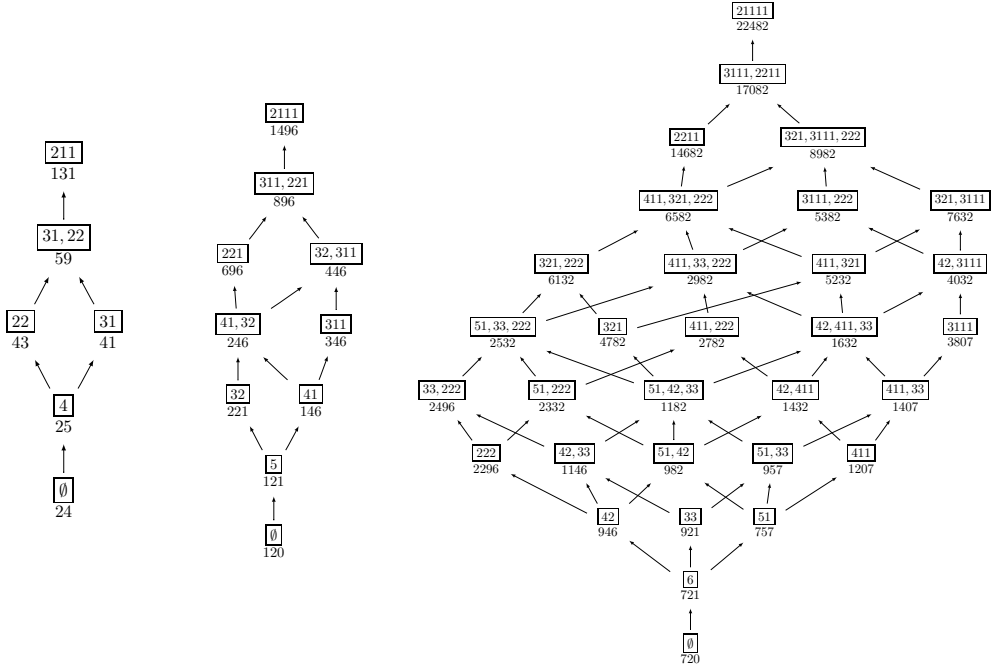

%%%%%%%%%%%%%%%%%%%%%%%%%%%%%%%%%%%%%%%%%%%%%%%%%%%%%%%%%
\section{Concluding remarks}
\label{section:remarks}
%%%%%%%%%%%%%%%%%%%%%%%%%%%%%%%%%%%%%%%%%%%%%%%%%%%%%%%%%

In this section, we state some interesting consequences of the characterization
of the lattice of submonoids of $\UU_k$ that contain $\sym_k$.

%%%%%%%%%%%%%%%%%%%%%%%%%%%%%%%%%%%%%%%%%%%%%%%%%%%%%%%%%
\subsection{The sizes of $\mathscr{J}$-classes are sums of squares of dimensions of irreducibles}
For monoids $M$, $N$ such that $\sym_k \subseteq M \subseteq N \subseteq \UU_k$,
the difference $N\backslash M$ is a union of $\mathscr{J}$-classes.  If $N$ covers
$M$ in the lattice of monoids, then $N = M \cup J_\mu$ for some partition $\mu$.

We note that $J_{1^k} = \sym_k$.
More generally, for $\lambda = 1^{a_1}2^{a_2}\cdots k^{a_k} \vdash k$,
let $\mathsf{sp}_k(\lambda)$ be the number of set partitions of $[k]$
such that the $\mathsf{type}$ is equal to $\lambda$, then
\begin{equation}
\label{eq:nosetpartitions}
    \mathsf{sp}_k(\lambda) = \frac{k!}{a_1! \cdots a_k!(1!)^{a_1}(2!)^{a_2} \cdots (k!)^{a_k}}~.
\end{equation}
Since every element of $J_\lambda$ can be considered as a bijection from a part of size $i$
in the top row to a part of size $i$ in the bottom row,
the number of elements of $J_\lambda$ is equal to $\mathsf{sp}_k(\lambda)^2 a_1! a_2! \cdots a_k!$~.

Recall from \cite{OSSZ.2022}, that the irreducible representations of $\UU_k$
are indexed by sequences of partitions,
$\vec{\lambda} = (\lambda^{(1)}, \lambda^{(2)}, \ldots, \lambda^{(k)})$,
such that $k = |\lambda^{(1)}|+ 2|\lambda^{(2)}| + \cdots + k|\lambda^{(k)}|.$
Let $I_k$ denote this set of sequences of partitions
and let $W_{\UU_k}^{\vec{\lambda}}$ represent the isomorphism class
of an irreducible $\UU_k$ module indexed by $\vec{\lambda} \in I_k$.

Now let $f^\lambda$ denote the number of standard tableaux
of shape $\lambda$ and define the
$\vtype(\vec{\lambda}) = 1^{|\lambda^{(1)}|}2^{|\lambda^{(2)}|}\cdots k^{|\lambda^{(k)}|}$.
Then~\cite[Corollary 3.19]{OSSZ.2022} states that
\begin{equation}
\label{equation.dimension}
\dim W_{\mathcal{U}_k}^{\vec{\lambda}} = \mathsf{sp}_k(\vtype(\vec{\lambda}))
f^{\lambda^{(1)}} f^{\lambda^{(2)}}\cdots f^{\lambda^{(k)}}~.
\end{equation}

\begin{prop}
For $\mu \vdash k$, we have
\[
	|J_\mu| = \sum_{\stackrel{\vec{\lambda}}{\vtype(\vec{\lambda}) = \mu}} (\dim W_{\mathcal{U}_k}^{\vec{\lambda}})^2~.
\]
\end{prop}
\begin{proof}
If $\mu = 1^{a_1}2^{a_2}\cdots k^{a_k}$, then $\vtype(\vec{\lambda}) = \mu$ implies that
$|\lambda^{(i)}| = a_i$.  Recall that $a! = \sum_{\lambda \vdash a} (f^\lambda)^2$.
Therefore by~\eqref{equation.dimension},
\[
	\sum_{\stackrel{\vec{\lambda}\in I_k}{\vtype(\vec{\lambda}) = \mu}} (\dim W_{\mathcal{U}_k}^{\vec{\lambda}})^2
	= \mathsf{sp}_k(\mu)^2 \sum_{\stackrel{\vec{\lambda}\in I_k}{\lambda^{(i)} \vdash a_i}} 
	(f^{\lambda^{(1)}})^2 (f^{\lambda^{(2)}})^2 \cdots (f^{\lambda^{(k)}})^2,
\]
where the sum on the right hand side is over all $\vec{\lambda} \in I_k$
such that $\lambda^{(i)}$ is a partition of $a_i$.
The conclusion that this expression is equal to $|J_\mu|$
follows from the enumeration $|J_\mu| = \mathsf{sp}_k(\mu)^2 a_1! a_2! \cdots a_k!$
of the elements of the $\mathscr{J}$-class described above.
\end{proof}

%%%%%%%%%%%%%%%%%%%%%%%%%%%%%%%%%%%%%%%%%%%%%%%%%%%%%%%%%
\subsection{The order $\preceq$ on partitions is not Cohen--Macaulay}
Ziegler~\cite{Ziegler.1986} proved that for $k \geqslant 19$, the poset
of integers ordered by refinement is not Cohen--Macaulay (as it had been previously
conjectured).  His proof produces an interval
of the poset at $k=19$ which is an obstruction
to the poset being Cohen--Macaulay and states that larger examples exist
at every $k$.

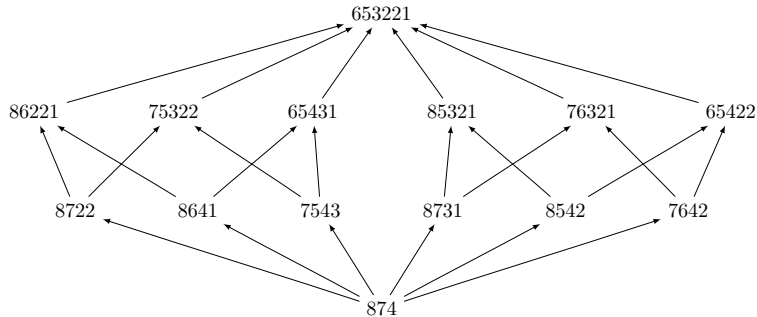
\begin{center}
\begin{scalebox}{.7}{
\begin{tikzpicture}[>=latex,line join=bevel,]
\node (node_0) at (215.5bp,8.5bp) [draw,draw=none] {$874$};
  \node (node_1) at (248.5bp,61.5bp) [draw,draw=none] {$8731$};
  \node (node_2) at (50.5bp,61.5bp) [draw,draw=none] {$8722$};
  \node (node_3) at (116.5bp,61.5bp) [draw,draw=none] {$8641$};
  \node (node_5) at (314.5bp,61.5bp) [draw,draw=none] {$8542$};
  \node (node_7) at (380.5bp,61.5bp) [draw,draw=none] {$7642$};
  \node (node_9) at (182.5bp,61.5bp) [draw,draw=none] {$7543$};
  \node (node_6) at (253.5bp,114.5bp) [draw,draw=none] {$85321$};
  \node (node_8) at (328.5bp,114.5bp) [draw,draw=none] {$76321$};
  \node (node_4) at (28.5bp,114.5bp) [draw,draw=none] {$86221$};
  \node (node_10) at (103.5bp,114.5bp) [draw,draw=none] {$75322$};
  \node (node_11) at (178.5bp,114.5bp) [draw,draw=none] {$65431$};
  \node (node_13) at (215.5bp,167.5bp) [draw,draw=none] {$653221$};
  \node (node_12) at (403.5bp,114.5bp) [draw,draw=none] {$65422$};
  \draw [black,->] (node_0) ..controls (225.0bp,24.182bp) and (232.12bp,35.181bp)  .. (node_1);
  \draw [black,->] (node_0) ..controls (168.9bp,23.904bp) and (118.37bp,39.521bp)  .. (node_2);
  \draw [black,->] (node_0) ..controls (184.88bp,25.274bp) and (158.96bp,38.628bp)  .. (node_3);
  \draw [black,->] (node_0) ..controls (246.12bp,25.274bp) and (272.04bp,38.628bp)  .. (node_5);
  \draw [black,->] (node_0) ..controls (262.1bp,23.904bp) and (312.63bp,39.521bp)  .. (node_7);
  \draw [black,->] (node_0) ..controls (206.0bp,24.182bp) and (198.88bp,35.181bp)  .. (node_9);
  \draw [black,->] (node_1) ..controls (249.9bp,76.805bp) and (250.91bp,87.034bp)  .. (node_6);
  \draw [black,->] (node_1) ..controls (272.83bp,78.01bp) and (292.86bp,90.78bp)  .. (node_8);
  \draw [black,->] (node_2) ..controls (44.232bp,77.031bp) and (39.621bp,87.72bp)  .. (node_4);
  \draw [black,->] (node_2) ..controls (66.149bp,77.558bp) and (78.403bp,89.35bp)  .. (node_10);
  \draw [black,->] (node_3) ..controls (89.607bp,78.086bp) and (67.287bp,91.021bp)  .. (node_4);
  \draw [black,->] (node_3) ..controls (134.99bp,77.709bp) and (149.72bp,89.823bp)  .. (node_11);
  \draw [black,->] (node_4) ..controls (88.41bp,131.84bp) and (142.1bp,146.48bp)  .. (node_13);
  \draw [black,->] (node_5) ..controls (296.31bp,77.709bp) and (281.82bp,89.823bp)  .. (node_6);
  \draw [black,->] (node_5) ..controls (341.83bp,78.161bp) and (364.7bp,91.264bp)  .. (node_12);
  \draw [black,->] (node_6) ..controls (242.45bp,130.33bp) and (234.02bp,141.65bp)  .. (node_13);
  \draw [black,->] (node_7) ..controls (365.15bp,77.558bp) and (353.12bp,89.35bp)  .. (node_8);
  \draw [black,->] (node_7) ..controls (387.05bp,77.031bp) and (391.87bp,87.72bp)  .. (node_12);
  \draw [black,->] (node_8) ..controls (293.3bp,131.39bp) and (263.15bp,145.0bp)  .. (node_13);
  \draw [black,->] (node_9) ..controls (158.47bp,78.01bp) and (138.69bp,90.78bp)  .. (node_10);
  \draw [black,->] (node_9) ..controls (181.38bp,76.805bp) and (180.58bp,87.034bp)  .. (node_11);
  \draw [black,->] (node_10) ..controls (138.39bp,131.39bp) and (168.27bp,145.0bp)  .. (node_13);
  \draw [black,->] (node_11) ..controls (189.21bp,130.26bp) and (197.3bp,141.41bp)  .. (node_13);
  \draw [black,->] (node_12) ..controls (343.27bp,131.84bp) and (289.29bp,146.48bp)  .. (node_13);
\end{tikzpicture}}
\end{scalebox}
\end{center}

We remark that this same interval also occurs in our poset and hence this poset will also not be Cohen--Macaulay.

Moreover, Ziegler~\cite{Ziegler.1986} proved that for $k\geqslant 111$ the M\"obius function
is not alternating. His proof also applies to our poset.

%%%%%%%%%%%%%%%%%%%%%%%%%%%%%%%%%%%%%%%%%%%%%%%%%%%%%%%%%
\subsection{The number of submonoids of $\UU_k$ that contain $\sym_k$}
Corollary~\ref{cor:submonoid characterization} can be used to compute the number of submonoids of
$\UU_k$ containing $\sym_k$
(i.e. $n_k := \#\{M: \sym_k \subseteq M \subseteq \UU_k \}$)
for small $k$
by computing the number of antichains
of the poset of partitions of $k$ not equal to $1^k$ under the order $\preceq$.
Using the poset functionality of {\sc SageMath}~\cite{sage}
we were able to compute the following table:
\begin{center}
%1,2,3,6,10,31,63,287,1099,8640,62658
\begin{scalebox}{.95}{
\begin{tabular}{c||c|c|c|c|c|c|c|c|c|c|c|c|c|c}
$k$&1&2&3&4&5&6&7&8&9&10&11&12&13&14\\
\hline
$n_k$&1&2&3&6&10&31&63&287&1099&8640&62658&1546891&29789119&2525655957
\end{tabular}}\end{scalebox}
\end{center}
The sequence has been submitted to \cite{OEIS} and is now sequence number
\href{https://oeis.org/A371505}{A371505}.

%%%%%%%%%%%%%%%%%%%%%%%%%%%%%%%%%%%%%%%%%%%%%%%%%%%
\bibliographystyle{plain}
\bibliography{main}{}

\end{document}